\documentclass[runningheads]{llncs}
\usepackage{graphicx}
\usepackage[symbol]{footmisc}

\usepackage{float}
\usepackage{mathtools}
\usepackage{amssymb, amsmath, latexsym}
\usepackage{nicefrac}
\usepackage{hhline}
\usepackage{makecell}
\usepackage{caption}
\usepackage{multirow}

\usepackage{colortbl}
\definecolor{bgcolor}{rgb}{0.8,1,1}
\definecolor{bgcolor2}{rgb}{0.8,1,0.8}
\usepackage{threeparttable}
\newcommand{\myred}[1]{{\color{red}#1}}

\usepackage{pifont}
\usepackage[colorinlistoftodos,bordercolor=orange,backgroundcolor=orange!20,linecolor=orange,textsize=scriptsize]{todonotes}
\usepackage{algorithm}\usepackage{algpseudocode}
\newcommand{\argmin}{\mathop{\arg\!\min}}

\allowdisplaybreaks
\graphicspath{ {images/} }
\newtheorem{assumption}{Assumption}

\newcommand{\eqdef}{\vcentcolon=}

\newcommand{\cO}{\mathcal{O}}

\DeclareMathOperator{\dom}{\mathrm{dom}}

\def\<#1,#2>{\langle #1,#2\rangle}

\begin{document} 
\title{Method with Batching 
for
Stochastic 
\\
Finite-Sum Variational
Inequalities 
\\
in Non-Euclidean Setting\thanks{The work of A. Pichugin and M. Pechin was supported by a grant for research centers in the field of artificial intelligence, provided by the Analytical Center for the Government of the Russian Federation in accordance with the subsidy agreement (agreement identifier 000000D730321P5Q0002) and the agreement with the Moscow Institute of Physics and Technology dated November 1, 2021 No. 70-2021-00138.}}
\titlerunning{Method with Batching for Stochastic VIs in Non-Euclidean Setting}
%
\author{Alexander Pichugin\inst{1}\and
Maksim Pechin\inst{1}\and
Aleksandr Beznosikov\inst{1,2,3} \and
Vasilii Novitskii\inst{1}\and 
Alexander Gasnikov\inst{3,1,2}
}
\authorrunning{M. Pechin, A. Pichugin, A. Beznosikov, A. Gasnikov}
%
\institute{Moscow Institute of Physics and Technology, Moscow, Russia, \and
Ivannikov Institute for System Programming of RAS, Moscow, Russia \and
Innopolis University, Innopolis, Russia
}
\maketitle              
\begin{abstract}
Variational inequalities are a universal optimization paradigm that incorporate classical minimization and saddle point problems. Nowadays more and more tasks require to consider stochastic formulations of optimization problems. In this paper, we present an analysis of a method that gives optimal convergence estimates for monotone stochastic finite-sum variational inequalities. In contrast to the previous works, our method supports batching, does not lose the oracle complexity optimality and uses an arbitrary Bregman distance to take into account geometry of the problem. Paper provides experimental confirmation to algorithm's effectiveness.
\end{abstract}

\section{Introduction}
In this paper, we consider the following variational inequality (VI) problem: 
\begin{equation}
    \label{eq:VI}
    \text{Find } x^* \in \mathcal{X} \text{ such that } \langle F(x^*), x - x^* \rangle + g(x) - g(x^*) \geq 0, \text{ for all } x \in \mathcal{X},
\end{equation}
where $F$ is some operator, $g$ is a proper convex lower semicontinuous function with domain $\operatorname{dom} g$. Such problem is the classical formulation of a variational inequality  \cite{stampacchia1964formes}. Moreover, we use the composite scheme for which $g$ is responsible. These kinds of variational inequalities have a wide usage.

\textbf{Application of variational inequalities.} 
Variational inequalities found significant implementations  across various disciplines.  In particular, they are used for classical problems such as equilibrium theory, games and economics \cite{NeumannGameTheory1944,HarkerVIsurvey1990,VIbook2003,rudeva2007variational}. 
Later, the machine learning community discovered a wide range of implications such as reinforcement learning \cite{omidshafiei2017deep,jin2020efficiently}, GANs (notably, Wasserstein GANs \cite{arjovsky2017wasserstein}), adversarial training \cite{madry2017towards}, supervised learning, unsupervised learning, discriminative clustering \cite{joachims2005support}, matrix factorization \cite{bach2008convex} and robust optimization \cite{ben2009robust}.
Some of these problems specfically require to utilize a stochastic setup of variational inequalities \cite{madry2017towards,sadiev2023high,beznosikov2023bregman}. Also stochastic variational inequalities are an instrumental in addressing several fundamental challenges, including balance problems in finances, management, and engineering contexts \cite{shapiro2021lectures,sun2021two,you2023prediction}.

\textbf{Our contribution and related works.} The simplest method for solving variational inequalities is the classical Gradient method. Unfortunately, it converges only for the cases of strongly monotone $F$, in the non-strongly monotone case convergence is lacking \cite{chavdarova2022continuous}. 
A big improvement in algorithms for VIs happened when the Extragradient method was proposed \cite{korpelevich1976extragradient}. It used the idea of the extrapolation for the Gradient method. Remarkably, it solved the convergence issue of the Gradient method.
Another big step was done in \cite{popov1980modification}, where the Optimistic method was proposed.
It is worth noticing a significant difference in these two methods, as Extragradient required us to call the oracle twice, while the algorithm from \cite{popov1980modification} is a single-call scheme. Later, A. Nemirovsky proposed exploiting the Bregman divergence for variational inequalities \cite{nemirovski2004prox}. This approach allowed to take into account generalized geometry that could be non-Euclidean. As the result, Mirror Prox algorithm was created. Next, there was a number of stochastic modifications of Extragradient and Mirror Prox.

Juditsky et al. \cite{juditsky2011solving} for the first time considered the stochastic version of Mirror Prox. This paper examines the general stochastic case with bounded variance. Another approach would be to consider a frequently used in practice finite-sum stochastic setup of the problem, where we can achieve better convergence rate due to its specifics. In particular, one of the earliest solution only for the saddle point problem in the Euclidean setup was considered by B. Palaniappan and F. Bach
\cite{palaniappan2016stochastic}. They based their method on the SVRG (Stochastic Variance reduction) algorithm \cite{johnson2013accelerating} and added the Catalyst envelope acceleration. As mentioned, the result was made for strongly convex-concave saddle point problems, which corresponds to strongly monotone variational inequations. Other algorithms utilizing the variance reduction technique (VR) for solving finite-sum variational inequalities with strongly monotone operators with better convergence rate are L-SVRGDA \cite{beznosikov2023stochastic}  and SVRE \cite{chavdarova2019reducing}. The idea of SVRE was continued in \cite{yang2020global} with the introduction of VR-AGDA. Here results were obtained in Polyak-Lojasiewicz conditions. Moreover, in \cite{tominin2021accelerated}, strongly convex-concave saddle point problems were considered. In this work, methods were designed with the Catalyst envelope acceleration. Achieved rate is comparable to \cite{palaniappan2016stochastic}.
In addition, it is worth mentioning the work \cite{carmon2019variance}. This paper deals with the Bregman setup and adapts the variance reduction technique. Unfortunately, this solution is limited by additional assumptions on the operator $F$ and by considering the matrix games setup. Either papers that friutfully implement VR techniques suitable for finite-sum setup are \cite{NIPS2014_ede7e2b6,allen2016improved,allen2018katyusha}; among others in the non-Euclidean case \cite{lan2019unified}. 
Malitsky and Alacaoglu in \cite{alacaoglu2021stochastic} applied variance reduction to Mirror Prox using double loop SVRG technique \cite{allen2016improved}.

Works above have one common problem -- proposed algorithms are sensitive to the size of mini-batches, apart from paper \cite{juditsky2011solving}. However, convergence rate in \cite{juditsky2011solving} does not reach the lower bound \cite{han2021lower}, because the general case is considered.
There are papers solve this problem for the strongly monotone operators \cite{kovalev2022optimal} and for the monotone operators \cite{pichugin2024optimal} in the Euclidean setup .

We compare the mentioned algorithms and their convergence results in Table \ref{tab:comparison}.
\renewcommand{\arraystretch}{2}
\begin{table}
    \centering
    \caption{Summary complexities for finding an $\varepsilon$-solution for monotone stochastic (finite-sum) variational inequality \eqref{eq:VI}+\eqref{eq:sum} with Lipschitz operators. Convergence is measured by gap function. {\em Notation:} $L$ and $L_2$ are Lipschitz constants for $F$ and $F_{m}$ in terms of $\|\cdot\|_*$ and $\|\cdot\|_2$ norms, respectively (see Assumptions \ref{ass3a}, \ref{ass3b}), $M$ = size of dataset, $b$ = batch size per iteration.}
    \label{tab:comparison}
    \small
  \begin{threeparttable}
    \begin{tabular}{|c|c|c|c|}
    \hline
    \quad\quad\quad\quad\quad\quad \textbf{Reference } \quad\quad\quad\quad\quad\quad & \textbf{Complexity} & \textbf{Non-Euclidean}
    \\\hline
    Nemirovski et al. \cite{nemirovski2004prox}\tnote{{\color{blue}(1)}} & $\mathcal{O} \left( \myred{M} \frac{L}{\varepsilon} \right)$  & \checkmark \\\hline
    Juditsky et al. \cite{juditsky2011solving} \tnote{\color{blue}(4)} & $\mathcal{O}\left( \frac{L}{\varepsilon} + \frac{1}{\varepsilon^\myred{2}} \right)$ & $\checkmark$ \\\hline
    Palaniappan \& Bach  \cite{palaniappan2016stochastic} \tnote{{\color{blue}(2,3,4)}} & $\mathcal{\tilde O} \left( \myred{b}\frac{L_2^2}{\varepsilon^\myred{2}}\right)$ & \\\hline
    Palaniappan \& Bach  \cite{palaniappan2016stochastic} \tnote{{\color{blue}(2,3,4)}}  & $\mathcal{\tilde O} \left( \sqrt{\myred{b}M}\frac{L_2}{\varepsilon} \right)$ &    \\\hline
    Chavdarova et al. \cite{chavdarova2019reducing} \tnote{{\color{blue}(3)}} & $\mathcal{\tilde O}\left(   \frac{\myred{b} L_2^{\textcolor{red}{2}}}{\varepsilon^\myred{2}}  \right)$ &   \\\hline
    Carmon et al. \cite{carmon2019variance} \tnote{\color{blue}(2,4)} & $\mathcal{\tilde O} \left( \sqrt{\myred{b}M}\frac{L}{\varepsilon} \right)$ & \checkmark \\\hline
    Yang et al. \cite{yang2020global} \tnote{\color{blue}(2,3,4)} & $\mathcal{\tilde O}\left( \myred{b^\frac{1}{3}M^\frac{2}{3}}\frac{L_2^{\textcolor{red}{3}}}{\varepsilon^\myred{3}}  \right)$ &  \\\hline
    Alacaoglu \& Malitsky \cite{alacaoglu2021stochastic} & $\mathcal{O} \left( \sqrt{\myred{b}M}\frac{L}{\varepsilon}\right)$ & \checkmark \\\hline
    Tominin et al. \cite{tominin2021accelerated} \tnote{\color{blue}(2,3,4)} & $\mathcal{\tilde O}\left( \sqrt{\myred{b}M}\frac{L_2}{\varepsilon} \right)$ &   \\\hline
    Beznosikov et al. \cite{beznosikov2023stochastic} \tnote{\color{blue}(3,4)} & $\mathcal{\tilde O}\left( \myred{b}\frac{L_2^{\textcolor{red}{2}}}{\varepsilon^\myred{2}} \right)$ &   \\\hline
    Kovalev et al. \cite{kovalev2022optimal} \tnote{\color{blue}(3,4)}  & $\mathcal{\tilde O}\left( \sqrt{M}\frac{L_2}{\varepsilon} \right)$ &  \\\hline
    
    Pichugin et al. \cite{pichugin2024optimal} \tnote{\color{blue}(5)} & $\mathcal{\tilde O}\left(   \sqrt{M}\frac{L_2}{\varepsilon}\right)  $ &    \\\hline
    \cellcolor{bgcolor2}{This paper} \tnote{\color{blue}(5)} & \cellcolor{bgcolor2}{$\mathcal{O} \left( \sqrt{M}\frac{L_2}{\varepsilon} \right)$} & \cellcolor{bgcolor2}{\checkmark}\\ \hline
    \cellcolor{bgcolor2}{This paper} \tnote{\color{blue}(5)} & \cellcolor{bgcolor2}{$\mathcal{\tilde O} \left( \sqrt{M}\frac{L}{\varepsilon} \right)$} & \cellcolor{bgcolor2}{\checkmark} \\\hline \hline
    Han et al.(lower bounds) \cite{han2021lower} & $\Omega \left( \sqrt{M}\frac{L_2}{\varepsilon}\right)$ & 
    \\\hline 
    \end{tabular}   
    \begin{tablenotes}
    \scriptsize
    \item   
    {\color{blue}(1)} deterministic methods, similar results were also obtained in \cite{tseng2000modified,mokhtari2020unified},
    \\{\color{blue}(2)} for saddle point problems only,
    \\{\color{blue}(3)} only for $\mu$-strongly monotone operators. To compare algorithms, we apply regularization trick,
    \\{\color{blue}(4)} only for bounded domain,
    \\{\color{blue}(5)} for $b \leq \sqrt{M}$.
\end{tablenotes}    
    \end{threeparttable}
\end{table}

In the present paper, our approach is a natural extension of the results from \cite{pichugin2024optimal}. On the contrary, here we consider an arbitrary non-Euclidean space. Correspondingly, our method is based on the following ideas: optimistic scheme \cite{popov1980modification} with negative momentum \cite{kovalev2022optimalvi} to deal with variance reduction and avoid double update like in Extragradient; on double-loop VR scheme to taking into account the Bregman divergence \cite{alacaoglu2021stochastic}. The oracle complexity of the new method is independent of batching, as long as the size of the batch does not exceed the square root of the full sample size.

\section{Main part}
In this paper, we assume that $\mathcal{X}\subseteq \mathbb{R}^n$ is a normed vector space with a dual space $\mathcal{X}^*$ and primal-dual norm pair $\| \cdot \|$ and $\| \cdot  \|_*$. As $\|\cdot \|$, we are using the norm $\|\cdot \|_p$ for $p\in [1,2]$.

We expect a stochastic nature of a problem: $F(x) = \mathbb{E}_{\xi \sim \mathcal{D}}[F_{\xi}(x)]$. As noted above, due to the fact that distribution $\mathcal{D}$ is non-trivial or even unknown, which appears in a number of applications, we can represent this as finite-sum approximation:
\begin{equation}
    \label{eq:sum}
    F(x) = \frac{1}{M} \sum\limits_{m=1}^M F_m (x),
\end{equation}
using the Monte-Carlo approach \cite{leluc2023monte}. In machine learning problems, the term “empirical risk” is often encountered. Note that calls of the full operator are expensive in practice. Thus, in order to avoid frequent computing operator $F$, one can use calls of single $F_m$ or mini-batches of them.

Desire to work in arbitrary geometry requires us to introduce an alternative way to mesure distance on the space rather than default Euclidean distance.

\begin{definition} Let $h:\; \mathcal{X}\xrightarrow{}\mathbb{R}\cup \{ +\infty\}$ be a convex lower semicontinuous function such that $\operatorname{dom} g \subseteq \operatorname{dom} h$ and $h$ is differentiable on $\operatorname{dom} h$, $h$ is 1-strongly convex on $\operatorname{dom} g$. We define the Bregman distance $V$: $\operatorname{dom} g \times \operatorname{dom} \partial h\xrightarrow{} \mathbb{R}_+$ associated with $h$ by
\begin{align*}
    V(u, v) = h(u)-h(v)-\langle \nabla h(v), u-v \rangle.
\end{align*}
\end{definition}                   
Here we used the following definition.
\begin{definition}
Function $f$ is 1-strongly convex if it can be lower bounded by a quadratic function of the form:
\begin{align*}
    f(y) \geq f(x) + \left\langle y-x, \nabla f(x) \right\rangle + \frac{1}{2}\|y-x\|^2
\end{align*}
for all $x$ and $y$ in $\dom f$.
\end{definition}
We also provide a couple of examples of $h$ and their corresponding $V$:
\begin{itemize}
    \item For $h(x) = \frac{1}{2}\|x\|_2^2$ we have $V(x, y) = \frac{1}{2}\|x-y\|_2^2$.
    \item The entropy function 
    \begin{align}
        h(x) = \sum_{i=1}^n x_i \log x_i \label{entropy}
    \end{align} in a probabilistic simplex 
    \begin{align}
        \Delta^n = \{ x\in \mathbb{R}^n\; | x_i\geq 0\; \sum_{i=1}^n x_i = 1\} \label{simplex}
    \end{align} generates KL-divergence
    \begin{align*}
        V(x,y) = \sum\limits_{i=1}^n x_i \operatorname{log}\frac{x_i}{y_i}.
    \end{align*}
\end{itemize}

Let us highlight the remarkable property of the Bregman distance which we use for the analyzis of convergence of our algorithm. Since $h$ is 1-strongly convex with respect to norm $\|\cdot \|$, we have for all $u, v \in \mathcal{X}$
\begin{align}
V(u,v) \geq \frac{1}{2} \|u-v\|^2. \label{main_inequality}
\end{align}
To further analyze the problem, we introduce the following assumptions:
\begin{assumption}\label{ass1} 
 The solution (maybe not unique) for the problem \eqref{eq:VI}+\eqref{eq:sum} exists.
\end{assumption}
\begin{assumption}\label{ass2}
    The operator $F$ is monotone, i.e. for all $u, v \in \mathcal{X}$ we have
\begin{equation*}
\langle F(u) - F(v); u - v \rangle \geq 0.
\end{equation*}
\end{assumption}
\captionsetup{type=assumption}
\renewcommand{\theassumption}{3a}
\begin{assumption}\label{ass3a}
    The operator $F$ is $L_2$-Lipschitz and for all $m\in [1,M]$ $F_m$ is $L_{2,m}$-Lipschitz, i.e. for all $u, v \in \mathcal{X}$ we have
\begin{align*}
\| F(u) - F(v) \|_2 &\leq L_2 \| u - v \|_2,
\\ \| F_m(u) - F_m(v) \|_2 &\leq L_{2,m} \| u - v \|_2.
\end{align*}
\end{assumption}
In addition, we define $\overline{L}_2$ such that $\overline{L}_2^2 = \frac{1}{M}\sum_{m=1}^{M}L^2_{2, m}$. With this notation, for all $u, v \in \mathcal{X}$ we have
\begin{align*}
    \frac{1}{M}\sum\limits_{m=1}^M \|F_m(u)-F_m(v)\|^2_2\leq \bar{L}^2_2 \|u-v\|^2_2  \label{barL2_def}.
\end{align*}
\captionsetup{type=assumption}
\renewcommand{\theassumption}{3b}
\begin{assumption}\label{ass3b}
The operator $F$ is $L$-Lipschitz and for all $m\in [1,M]$ $F_m$ is $L$-Lipschitz, i.e. for all $u, v \in \mathcal{X}$ we have
\begin{align*}
\| F(u) - F(v) \|_* &\leq L \| u - v \|,
\\ \| F_m(u) - F_m(v) \|_* &\leq L \| u - v \|.
\end{align*}
\end{assumption}
Note that $\|\cdot\|\leq \|\cdot\|_2$ for all $p\in [1,2]$. Thus, we can state that $L\leq \bar L_2$. In addition, we would like to point out that these meanings often differ dramatically.

\subsection{Algorithm}
Now let us state the following
algorithm:

\begin{algorithm}[H]
	\caption{Optimistic Method with Momentum and Batching}
	\label{alg:new}
	\begin{algorithmic}[1]
\State
\noindent {\bf Parameters:}  stepsize $\eta>0$, momentum $\gamma > 0$, probability $p \in (0;1)$, batch size $b \in \{1,\ldots,M\}$, number of iterations $K$\\
\noindent {\bf Initialization:} choose  $x^{-1}_0 = x^0_0 = w^0_0 = x^k_{-1} = w^k_{-1} \in \mathcal{X}$ for all $k \in [-1, K-1]$ \label{Alg1Line2}
\For {$s = 0, 1, \ldots S-1$}
\For {$k=0, 1, \ldots, K - 1$ }
    \State Sample $j_1^k,\ldots, j_b^k$ independently from $\{1,...,M\}$ uniformly at random \label{Alg1Line5}
    \State $B^k=\{j_1^k,..., j_b^k\}$\label{Alg1Line6}
    \State $\Delta^k = \frac{1}{b}\sum\limits_{j\in B^k}(F_j(x^k_s)-F_j(w_s)+(F_j(x^k_s)-F_j(x^{k-1}_s)))+F(w_s)$ \label{Alg1Line7}
    \State $x^{k+1}_s=\argmin\limits_{x\in \mathcal{X}} \left( g(x) + \frac{1}{\eta}(1-\gamma)V(x,x^k_s)+ \frac{1}{\eta}\cdot\gamma V(x,\overline{w}_s) + \langle \Delta^k, x\rangle \right)$ \label{Alg1Line8}
    \EndFor
    \State $w_{s+1}= \frac{1}{K}\sum\limits_{k=1}^K x^k_s$\label{Alg1Line10}
    \State $\nabla h(\overline{w}_{s+1})= \frac{1}{K}\sum\limits_{k=1}^K \nabla h(x^k_s)$\label{Alg1Line11}
    \State $x^0_{s+1}= x_s^K$\label{Alg1Line12}
    \State $x^{-1}_{s+1} = x_s^{K-1}$\label{Alg1Line14}

\EndFor
	\end{algorithmic}
\end{algorithm}

Note that in line \ref{Alg1Line7} of Algorithm \ref{alg:new} we use the variance reduction technique in $\left(F_j(x^k_s)-F_j(w_s)+F(w_s)\right)$ part. The key difference from Stochastic Gradient Descent (SGD) is that instead of using $g_k = F_j(x^k)$, variance reduced methods use the approximation $g_k = F(w^k) + F_j(x^k) - F_j(w^k)$. This helps to decrease "variance" $\mathbb{E}\|g_k - F(x_k)\|^2$ comparing to SGD in the case of "good" choice of $\omega^k$ \cite{johnson2013accelerating,allen2017katyusha,alacaoglu2021stochastic}. We also add batching in lines \ref{Alg1Line5} and \ref{Alg1Line7}.

In addition, in line \ref{Alg1Line7} of Algorithm \ref{alg:new} we use $\left(F_j\left(x^k_s\right)-F_j\left(x^{k-1}_s\right)\right)$ term to implement so-called optimistic scheme, slightly different from the original option \cite{popov1980modification}. Our update is a modification of Forward-Reflected-Backward approach \cite{malitsky2020forward}, where we set $\alpha \equiv 1$ in $x^{k+1} = \eta F_j(x^k) + \eta\alpha\left[F_j(x^k) - F_j(x^{k-1})\right]$.

While for the minimization problems it is usual to apply positive (heavy-ball) momentum \cite{polyak1987introduction}, the opposite approach turns out to be suitable for variational inequalities. This effect was noticed earlier \cite{gidel2019negative,yoon2021accelerated,alacaoglu2021forward} and appeared now in the theory of stochastic methods for VIs. Hence, in line \ref{Alg1Line8} we also apply the negative momentum in $\left(-\frac{1}{\eta}\gamma V(x,x^k_s)+ \frac{1}{\eta}\cdot\gamma V(x,\overline{w}_s)\right)$ part. For illisutation, in the Euclidean case the negative momentum would have looked like $\gamma (w^k-x^k)$.

The snapshot point updates in line \ref{Alg1Line10} similar to \cite{alacaoglu2021stochastic} and SVRG \cite{johnson2013accelerating}. Due to the Bregman setup is that we have the additional point $\overline{w}_{s+1}$ that averages in the dual space.

Now we are ready to proof convergence of Algorithm \ref{alg:new}.

\subsection{Analysis}
In order to calculate a convergence rate, we use the gap function \cite{nemirovski2004prox,juditsky2011solving} as a standard convergence criterion for such problems:
\begin{equation}
    \label{gap}
    \text{Gap} (z) \eqdef \sup_{u \in \mathcal{C}} \left[ \langle F(u),  z - u  \rangle + g(x) - g(u)\right].
\end{equation}
Here $\mathcal{C}$ is a compact subset of $\mathcal{X}$ used to handle the case if  $\operatorname{dom}g$ is unbounded [see Lemma 1 from \cite{nesterov2007dual}].

To begin working with update, we propose the following lemmas.

\begin{lemma}\label{lem:1} [See Lemma 3.2 from \cite{alacaoglu2021stochastic}]

Let $g$ be a proper convex lower semicontinuous function. Denote
\begin{equation*}
    x^\dag = \underset{x \in \mathcal{C}}{\operatorname{argmin}}\left(g(x) + \langle u, x \rangle + \gamma V(x, x_1) + (1-\gamma)V(x, x_2)\right).
\end{equation*}
Then, for all $x$ it delivers
\begin{align*}
    g(x) - g(x^\dag) &+ \langle u, x - x^\dag \rangle 
    \\\geq& V(x, x^\dag) + \gamma(V(x^\dag, x_1) - V(x, x_1)) + (1-\gamma)(V(x^\dag,x_2)-V(x, x_2)).
\end{align*}
\end{lemma}
\begin{proof}
The proof of this lemma directly follows from the first order optimality condition \cite{boyd2004convex}.    
\end{proof}

Hereafter, in the next lemma, we estimate the variance of $\Delta^k$. 

\captionsetup{type=lemma}
\renewcommand{\thelemma}{2a}
\begin{lemma}\label{lem:2a}[See Lemma 2 from \cite{pichugin2024optimal}]
For any step $s$ from 0 to $S-1$ and for any $k$  from 0 to $K-1$ of Algorithm \ref{alg:new} under the Assumption \ref{ass3a} the following inequality holds: 
\begin{equation*}
    \mathbb{E}\left[\left\|\Delta^{k}-\mathbb{E}_{k}\left[\Delta^{k}\right]\right\|^{2}_2\right]
    \leq{\frac{2\bar{L}^2_2}{b}}\mathbb{E}\left[\|x_s^{k}-w_s\|^{2}+\left\|x_s^{k}-x_s^{k-1}\right\|^{2}\right], 
\end{equation*}
where $\mathbb{E}_k\left[\Delta^k\right]$ is equal to 
\begin{align}
    \mathbb{E}_{k}\left[\Delta^{k}\right]=2F(x_s^{k})-F(x_s^{k-1}).
    \label{lem2a:Ek}
\end{align}
\end{lemma}
\begin{proof}
We start from line \ref{Alg1Line7} of Algorithm \ref{alg:new}, 
\begin{align*}
\mathbb{E}_k\Big[\big\|\Delta^k -\mathbb{E}_k\left[\Delta^k\right]\big\|^2_2\Big] 
 = &
 \mathbb{E}_k\Bigg[\bigg\|\frac{1}{b}\sum\limits_{j\in B^k}(F_j(x_s^k)-F_j(w_s)+(F_j(x_s^k)-F_j(x_s^{k-1})))
 \\&+F(w_s)- (2 F(x_s^k) - F(x_s^{k-1}))\bigg\|^2_2\Bigg]. 
\end{align*}
With Cauchy–Schwarz inequality (\ref{ap:Cauchy–Schwarz}), we have
\begin{align*}
\mathbb{E}_k\Big[\big\|\Delta^k & -\mathbb{E}_k\left[\Delta^k\right]\big\|^2_2\Big]
\\
\leq &
2 \mathbb{E}_k\left[\left\|\frac{1}{b} \sum_{j \in B^k}\left(F_j\left(x_s^k\right)-F_j\left(w_s\right)\right)-\left(F\left(x_s^k\right)-F\left(w_s\right)\right)\right\|^2_2\right] \notag
\\
& +2 \mathbb{E}_k\left[\left\|\frac{1}{b} \sum_{j \in B^k}\left(F_j\left(x_s^k\right)-F_j\left(x_s^{k-1}\right)\right)-
\left(F\left(x_s^k\right)-F\left(x_s^{k-1}\right)\right)\right\|^2_2\right].
\end{align*}
Using that we choose $j_1^k,\ldots, j_s^k$ in $B^k$ indepdently and uniformly, one can note
 \begin{align*}
&\mathbb{E}_k \Big[\big\langle \left( F_j\left(x^k_s\right)-F_j\left(w_s\right)\right)-\left(F\left(x^k_s\right)-F\left(w_s\right) \right), \left( F_j\left(x^k_s\right)-F_j\left(w_s\right)\right)
\\
&\hspace{4cm}-\left(F\left(x^k_s\right)-F\left(w_s\right) \right)\big\rangle \Big]
\\
&\hspace{3cm}=\mathbb{E}_k \bigg[\langle \mathbb{E}_{j^k_i} \left[\left( F_{j^k_i}\left(x^k_s\right)-F_{j^k_i}\left(w_s\right)\right)-\left(F\left(x^k_s\right)-F\left(w_s\right) \right) \right], 
\\
&\hspace{4cm}\mathbb{E}_{j^k_l} \left[\left( F_{j^k_l}\left(x^k_s\right)-F_{j^k_l}\left(w_s\right)\right)-\left(F\left(x^k_s\right)-F\left(w_s\right) \right) \right]\rangle \bigg]
\\
&\hspace{3cm}=0.
\end{align*}
Hence, we get
 \begin{align*}
 \mathbb{E}_k\Big[\big\|\Delta^k & -\mathbb{E}_k\left[\Delta^k\right]\big\|^2_2\Big]
 \\
\leq&2 \mathbb{E}_k\left[\sum_{j \in B^k}\frac{1}{b^2}\left\|\left(F_j\left(x^k_s\right)-F_j\left(w_s\right)\right)-\left(F\left(x^k_s\right)-F\left(w_s\right)\right)\right\|^2_2\right] \notag
  \\
&+ 2 \mathbb{E}_k\left[\sum_{j \in B^k}\frac{1}{b^2}\left\|\left(F_j\left(x^k_s\right)-F_j\left(x^{k-1}_s\right)\right)-\left(F\left(x^k_s\right)-F\left(x^{k-1}\right)\right)\right\|^2_2\right] \notag
\\
=&\frac{2}{b^2} \mathbb{E}_k\left[\sum_{j \in B^k}\left\|\left(F_j\left(x^k_s\right)-F_j\left(w_s\right)\right)-\left(F\left(x^k_s\right)-F\left(w_s\right)\right)\right\|^2_2\right] \notag
  \\
&+\frac{2}{b^2} \mathbb{E}_k\left[\sum_{j \in B^k}\left\|\left(F_j\left(x^k_s\right)-F_j\left(x^{k-1}_s\right)\right)-\left(F\left(x^k_s\right)-F\left(x^{k-1}_s\right)\right)\right\|^2_2\right] \notag
\\
 \leq & \frac{2}{b^2} \mathbb{E}_k\left[\sum_{j \in B^k}\left\|F_j\left(x^k_s\right)-F_j\left(w_s\right)\right\|^2_2\right]
+\frac{2}{b^2} \mathbb{E}_k\left[\sum_{j \in B^k}\left\|F_j\left(x^k_s\right)-F_j\left(x^{k-1}_s\right)\right\|^2_2\right]. \notag
\end{align*}
In the last step, we used the fact that $\mathbb{E}\|X-\mathbb{E}X\|^2 = \mathbb{E}\|X\|^2-\|\mathbb{E}X\|^2$. Next, we again take into account that $j_1^k,\ldots, j_s^k$ in $B^k$ are chosen uniformly,
 \begin{align*}
 & \mathbb{E}_k\Big[\big\|\Delta^k  -\mathbb{E}_k\left[\Delta^k\right]\big\|^2_2\Big]
 \\
 &\hspace{0.7cm} \leq \frac{2}{b} \mathbb{E}_k\left[\mathbb{E}_{j \sim \text { u.a.r. }\{1, \ldots, M\}}\left[\left\|F_j\left(x^k_s\right)-F_j\left(w_s\right)\right\|^2_2 +\left\|F_j\left(x^k_s\right)-F_j\left(x^{k-1}_s\right)\right\|^2_2\right]\right] \notag
 \\
 &\hspace{0.7cm} = \frac{2}{M b} \sum_{j=1}^M \left(\left\|F_j\left(x^k_s\right)-F_j\left(w_s\right)\right\|^2_2+\left\|F_j\left(x^k_s\right)-F_j\left(x^{k-1}_s\right)\right\|^2_2\right) .
\end{align*}
We take the full expectation of both parts:
\begin{align*}
 \mathbb{E}\Big[\big\|\Delta^k & -\mathbb{E}_k\left[\Delta^k\right]\big\|^2_2\Big]
\leq \frac{2}{M b} \mathbb{E}\left[\sum_{j=1}^M \left(\left\|F_j\left(x^k_s\right)-F_j\left(w_s\right)\right\|^2_2+\left\|F_j\left(x^k_s\right)-F_j\left(x^{k-1}_s\right)\right\|^2_2\right)\right] .
\end{align*}
Using $L_2$-Lipschitzness of $F$ (Assumption \ref{ass3a}) and the fact that $\| \cdot \|_2\leq \| \cdot \|$, we can rewrite it as
\begin{align*}
    \mathbb{E}\left[\left\|\Delta^{k}-\mathbb{E}_{k}\left[\Delta^{k}\right]\right\|^2_2 \right]\leq & {\frac{2\bar{L}^2_2}{b}}\mathbb{E}\left[\|x^k_s-w_s\|^{2}_2+\left\|x^k_s-x^{k-1}_s\right\|^{2}_2\right]
    \\ \leq & {\frac{2 \bar{L}^2_2}{b}}\mathbb{E}\left[\|x^k_s-w_s\|^{2}+\left\|x^k_s-x^{k-1}_s\right\|^{2}\right].
\end{align*}
This finishes the proof.

\end{proof}

\captionsetup{type=lemma}
\renewcommand{\thelemma}{2b}
\begin{lemma}\label{lem:2b} For any step $s$ from 0 to $S-1$ and for any $k$  from 0 to $K-1$ of Algorithm \ref{alg:new} under the Assumption \ref{ass3b} the following inequality holds:
\begin{equation*}
    \mathbb{E}\left[\left\|\Delta^{k}-\mathbb{E}_{k}\left[\Delta^{k}\right]\right\|^{2}_*\right]
    \leq{\frac{2(1+C\ln n)\bar{L}^2}{b}}\mathbb{E}\left[\|x_s^{k}-w_s\|^{2}+\left\|x_s^{k}-x_s^{k-1}\right\|^{2}\right], 
\end{equation*}
where $\mathbb{E}_k\left[\Delta^k\right]$ is equal to 
\begin{equation*}
    \mathbb{E}_{k}\left[\Delta^{k}\right]=2F(x_s^{k})-F(x_s^{k-1}).
\end{equation*}
\end{lemma}
\begin{proof}
    
Similar to the proof of Lemma \ref{lem:2a} we start from line \ref{Alg1Line7} of Algorithm \ref{alg:new}, 
\begin{align*}
\mathbb{E}\Big[\big\|\Delta^k & -\mathbb{E}\left[\Delta^k\right]\big\|^2_*\Big] 
\\ 
 = &
 \mathbb{E}_k\Bigg[\Bigg\|\frac{1}{b}\sum\limits_{j\in B^k}\big(F_j(x_s^k)-F_j(w_s)+(F_j(x_s^k)-F_j(x_s^{k-1})))+F(w_s)
 \\& - (2 F(x_s^k) - F(x_s^{k-1})\big)\Bigg\|^2_*\Bigg].
\end{align*}
For shortness we introduce the stochastic variables
\begin{align*}
    \xi^j = \left(F_j(x_s^k)-F_j(w_s)+(F_j(x_s^k)-F_j(x_s^{k-1})))+F(w_s) - (2 F(x_s^k) - F(x_s^{k-1})\right),
\end{align*}
and declare the following properties of them:
\begin{enumerate}
    \item\label{exp} $\mathbb{E} \xi^j = 0 \;\;\; \forall j\in B^k$,
    \item $\xi^j$ is independent,
    \item $\|\xi^j  \|<\sigma = \sqrt{2 \bar{L} \left( \|x_s^k - w_s\|^2 + \|x_s^k - x_s^{k-1}\|^2 \right)}$. (Here we used the Assumption \ref{ass3b}).
\end{enumerate}
We also bring in the sequence $\left\{u^j\right\}_{j\in B^k}$ defined as
\begin{align}
    u^{j+1} = \argmin_{y\in B_1(0)} \left( V(u, u^j) + \left\langle\frac{\Omega}{\sigma \sqrt{b}} \xi^j, y \right\rangle \right) 
    \label{u_def}
\end{align}
with a starting point $u^1 = 0$. In addition, we introduce $\max\limits_{u\in B_1(0)} V(u, u^1) = \max\limits_{u\in B_1(0)} V(u, 0) = \frac{\Omega^2}{2}$.

With these denotations and the definition of $\|\cdot \|_*$, it is correct that
\begin{align}
    \left\|\sum_{j\in B^k} \frac{\Omega}{\sigma \sqrt{b}} \xi^j\right\|_* =& 
    \max_{u\in B_1(0)} \left\langle \sum_{j\in B^k} \frac{\Omega}{\sigma \sqrt{b}}\xi^j, u \right\rangle. \notag
\end{align}
To estimate the right side, we can write the optimality condition \cite{boyd2004convex} for (\ref{u_def}):
\begin{align*}
    \left\langle \nabla h(u^{j+1})-\nabla h(u^j)-\frac{\Omega}{\sigma \sqrt{b}}\xi^{j+1},y-u^{j+1} \right\rangle\geq 0.
\end{align*}
We apply the three point identity (\ref{ap:3point}) to obtain
\begin{align}
    V(u, u^j)-V(y,u^{j+1})-V(u^{j+1}, u^j)- \left\langle \frac{\Omega}{\sigma \sqrt{b}}\xi^{j+1}, y-u^{j+1} \right\rangle\geq 0 .\label{lem2b:before_young}
\end{align}
Using Young’s inequality (\ref{ap:Young}) and the inequality (\ref{main_inequality}), we achieve
\begin{align}
\left\langle \frac{\Omega}{\sigma \sqrt{b}}\xi^{j+1}, y-u^{j+1} \right\rangle= & \left\langle \frac{\Omega}{\sigma \sqrt{b}}\xi^{j+1}, y-u^{j} \right\rangle + \left\langle \frac{\Omega}{\sigma \sqrt{b}}\xi^{j+1}, u^j-u^{j+1} \right\rangle \notag
    \\\geq & \left\langle \frac{\Omega}{\sigma \sqrt{b}}\xi^{j+1}, y-u^{j} \right\rangle + \frac{1}{2}\left\|\frac{\Omega}{\sigma \sqrt{b}}\xi^{j+1}\right\|^2_*-\frac{1}{2} \|u^{j+1}-u^j \|^2\notag
    \\\geq & \left\langle \frac{\Omega}{\sigma \sqrt{b}}\xi^{j+1}, y-u^{j} \right\rangle + \frac{1}{2}\left\|\frac{\Omega}{\sigma \sqrt{b}}\xi^{j+1}\right\|^2_*-V(u^{j+1},u^j). \label{lem2b:young}
\end{align}
Combining (\ref{lem2b:before_young}) with (\ref{lem2b:young}), we get
\begin{align*}
    \left\langle \frac{\Omega}{\sigma \sqrt{b}}\xi^{j+1}, y \right\rangle \leq V(y, u^j) - V(y, u^{j+1})+\left\langle \frac{\Omega}{\sigma \sqrt{b}}\xi^{j+1}, u^j\right\rangle + \frac{1}{2}\left\|\frac{\Omega}{\sigma \sqrt{b}}\xi^{j+1}\right\|^2_*.
\end{align*}
Now we take sum for all $j$ to estimate
\begin{align*}
    \left\|\sum_{j\in B^k} \frac{\Omega}{\sigma \sqrt{b}} \xi^j\right\|_* 
    \leq &\max_{u\in B_1(0)} V(u, u^1) + \frac{1}{2} \sum_{j\in B^k} \left\| \frac{\Omega}{\sigma \sqrt{b}}\xi^j \right\|^2_* + \sum_{j\in B^k} \left\langle \frac{\Omega}{\sigma \sqrt{b}}\xi^j, u^j \right\rangle\notag
    \\
    \leq & \frac{\Omega^2}{2} + \frac{1}{2} \sum_{j\in B^k} \left(\frac{\Omega}{\sqrt{b}}\right) ^2  + \sum_{j\in B^k} \left\langle \frac{\Omega}{\sigma \sqrt{b}}\xi^j, u^j \right\rangle. 
\end{align*}
Here we used the third property of $\xi^j$.
Then,
\begin{align*}
    \left\|\sum_{j\in B^k}  \xi^j\right\|_* \leq \sigma \Omega \sqrt{b} + \sum_{j\in B^k} \left\langle \xi^j, u^j \right\rangle.
\end{align*}
After taking square of both sides, one can get
\begin{align*}
    \left\|\sum_{j=1}^b  \xi^j\right\|_*^2 
    \leq & \sigma^2 \Omega^2 b + \sum_{j\in B^k} \left\langle \xi^j, u^j \right\rangle ^2 + 2 \sigma \Omega \sqrt{b} \sum_{j\in B^k} \left\langle \xi^j, u^j \right\rangle.
\end{align*}
Now we take the expectation of both sides
\begin{align}
\mathbb{E}\left\|\sum_{j\in B^k} \xi^j\right\|_*^2
\leq & \mathbb{E}\left[\sigma^2 \Omega^2 b\right] + \mathbb{E}\sum_{j\in B^k} \left\langle \xi^j, u^j \right\rangle ^2 + 2 \Omega \sqrt{b} \sum_{j\in B^k} \mathbb{E}\sigma\left\langle \xi^j, u^j \right\rangle.\label{lem2b:before_exp}
\end{align}
 We use the fact that for all $j$ from $B^K$ $\xi^j$ and $u^j$ are independent to get
 \begin{align}
     \sum_{j\in B^k} \mathbb{E}\sigma\left\langle \xi^j, u^j \right\rangle = 0. \label{lem2b:xu_u}
 \end{align}
Combining (\ref{lem2b:before_exp}) and (\ref{lem2b:xu_u}), one can achieve
\begin{align*}
\mathbb{E}\left\|\sum_{j\in B^k} \xi^j\right\|_*^2
\leq & \mathbb{E}\left[\sigma^2 \Omega^2 b\right] + \mathbb{E}\sum_{j\in B^k} \left\langle \xi^j, u^j \right\rangle ^2.
\end{align*}
Now what is left is to apply the Young's inequality (\ref{ap:Young}) and the first property of $\xi^j$ to get
\begin{align*}
\mathbb{E}\left\|\sum_{j\in B^k} \xi^j\right\|_*^2
\leq & \mathbb{E}\left[\sigma^2 \Omega^2 b\right] + \mathbb{E}\sum_{j\in B^k} \left[ \|\xi^j\|^2_*\cdot \|u^j\|^2 \right]
\\
\leq & \mathbb{E}\left[\sigma^2 \Omega^2 b\right] + \mathbb{E}\left[b \sigma^2 \right] 
\\
\leq & \mathbb{E}\left[\sigma^2 b (\Omega^2 + 1)\right].
\end{align*}
In order to perform the next step, we need to evaluate $\Omega$. This kind of result may be found in Table \ref{table_gasnikov} of \cite{gasnikov2017universal}. For convinience, it is represented below. 
\begin{table}[H]
    \centering
    \begin{tabular}{|c|c|c|c|}
    \hline
        $Q = B^n_p(1)$ & $1\leq p \leq a$ & $a \leq p \leq 2$ & $2 \leq p \leq \infty$ \\ \hline
        $\|\cdot\|$ & $\|\cdot\|_1$ & $\|\cdot\|_p$ & $\|\cdot\|_2$ \\ \hline
        $h(x)$ & $\frac{1}{2(a-1)}\|x\|^2_a$ & $\frac{1}{2(p-1)}\|x\|^2_p$ & $\frac{1}{2}\|x\|^2_2$ \\ \hline
        $\Omega^2$ & $\cO(\ln n)$ & $\cO((p-1)^{-1})$ & $\cO(n^{\frac{1}{2} - \frac{1}{p}})$ \\ \hline
    \end{tabular}
    \caption{Examples of prox-functions for ball-shaped sets $Q$ in various norms}
    \label{table_gasnikov}
\end{table}

In particular, we can choose $h(x) = \frac{1}{2(a-1)}\|x\|^2_a$ for $a = \frac{2 \ln n}{2\ln n - 1}$. At this case for $p\in [1,2]$ it is also true that $\Omega^2 = \cO(\ln n)$, where $n$ is dimensionality of $\mathcal{X}$. After plugging constant $\sqrt{C}$ in the definition of $\cO(\cdot)$ bound, substituting $\sigma = \sqrt{2 \bar{L} \left( \|x_s^k - w_s\|^2 + \|x_s^k - x_s^{k-1}\|^2 \right)}$, we finally get
\begin{align*}
    \mathbb{E}\left\|\frac{1}{b}\sum_{j\in B^k} \xi^j\right\|_*^2 \leq \frac{2\bar{L} (1+C\ln n)}{b}\mathbb{E}\left[\|x_s^k - w_s\|^2 + \|x_s^k - x_s^{k-1}\|^2\right].
\end{align*}
That allows us to finish the proof.
\end{proof}

\captionsetup{type=lemma}
\renewcommand{\thelemma}{3a}
\begin{lemma}\label{lem:3a}[see Lemma 2.4 from \cite{alacaoglu2021stochastic}]
    Let $\mathcal{F} = (\mathcal{F}_k)_{k\geq 0}$ be a filtration and $(u^k)$ a stochastic process adapted to $\mathcal{F}$ with $\mathbb{E}[u^{k+1}|\mathcal{F}_k]=0$. Then for any $K\in \mathbb{N}$, $x^0\in X$ and any compact set $\mathcal{C}\subseteq X$
    \begin{align*}
        \mathbb{E}\left[\max\limits_{x\in \mathcal{C}}\sum\limits_{k=0}^{K-1}\langle u^{k+1}, x \rangle\right] \leq \max\limits_{x\in \mathcal{C}}\frac{1}{2}\|x^0-x\|^2_2+\frac{1}{2}\sum\limits_{k=0}^{K-1}\mathbb{E}\|u^{k+1}\|^2_2.
    \end{align*}
\end{lemma}
\begin{proof}
    See proof from Lemma \ref{lem:3b} in the case of substitution $V(x, y) = \frac{1}{2}\|x - y\|^2$.
\end{proof}
\captionsetup{type=lemma}
\renewcommand{\thelemma}{3b}
\begin{lemma}\label{lem:3b}[see Lemma 3.5 from \cite{alacaoglu2021stochastic}]
Let $\mathcal{F} = (\mathcal{F}_k)_{k\geq 0}$ be a filtration and $(u^k)$ a stochastic process adapted to $\mathcal{F}$ with $\mathbb{E}[u^{k+1}| \mathcal{F}_k]=0$. Then for any $K\in \mathbb{N}$, $x^0\in X$ and any compact set $\mathcal{C}\subseteq X$    
\begin{align*}
    \mathbb{E}\left[ \max\limits_{x\in \mathcal{C}} \sum_{s=0}^{S-1}\sum_{k=0}^{K-1}\langle u^{k+1}, x \rangle \right]\leq \max_{x\in \mathcal{C}}V(x,x_0)+\frac{1}{2}\sum_{s=0}^{S-1}\sum_{k=0}^{K-1}\mathbb{E}\|u^{k+1}\|^2_*
\end{align*}

\end{lemma}
\begin{proof}
Let us define 
\begin{align*}
    z^{k+1} = \argmin_{x\in \dom g}\left\{ \langle-u^{k+1}, x \rangle + V(x,z^k)\right\}.
\end{align*}
We use the the first order optimality condition \cite{boyd2004convex} for this notation to get
\begin{align*}
    \langle \nabla h(z^{k+1})-\nabla h(z^k)-u^{k+1},x-z^{k+1} \rangle\geq 0.
\end{align*}
We apply the three point identity (\ref{ap:3point}) to obtain
\begin{align}
    V(x, z^k)-V(x,z^{k+1})-V(z^{k+1}, z^k)- \langle u^{k+1}, x-z^{k+1} \rangle\geq 0 .\label{lem:4:before_young}
\end{align}
Applying the Young’s inequality (\ref{ap:Young}), we achieve
\begin{align}
\langle u^{k+1}, x-z^{k+1} \rangle= & \langle u^{k+1}, x-z^{k} \rangle + \langle u^{k+1}, z^k-z^{k+1} \rangle \notag
    \\\geq & \langle u^{k+1}, x-z^{k} \rangle + \frac{1}{2}\|u^{k+1}\|^2_*-\frac{1}{2} \|z^{k+1}-z^k \|^2\notag
    \\\geq & \langle u^{k+1}, x-z^{k} \rangle + \frac{1}{2}\|u^{k+1}\|^2_*-V(z^{k+1},z^k). \label{lem:4:young}
\end{align}
Combining (\ref{lem:4:before_young}) with (\ref{lem:4:young}), we get
\begin{align*}
    \langle u^{k+1}, x \rangle \leq V(x, z^k)- V(x, z^{k+1})+\langle u^{k+1}, z^k\rangle + \frac{1}{2}\|u^{k+1}\|^2_*.
\end{align*}
We sum this inequality for all $k$ from $0$ to $K-1$, take maximum, expectation and use the fact that $\mathbb{E}\left[\sum\limits_{s=0}^{S-1}\sum\limits_{k=1}^{K-1}\langle u^{k+1}, z^k \rangle\right]=0$. That result finishes the proof.
\end{proof}

Finally, we introduce a simple technical lemma.
\captionsetup{type=lemma}
\renewcommand{\thelemma}{4}
\begin{lemma}\label{lem:4}(see (25) from \cite{alacaoglu2021stochastic}) In the conditions of Algorithm \ref{alg:new} it is  applied that 
\begin{align*}
V(x, \overline{\omega}_s)-V(x^{k+1}_s, \overline{\omega}_s) = \frac{1}{K}\sum_{j=0}^{K-1}\left (V(x, x_{s-1}^j)-V(x_s^{k+1},x_{s-1}^j)\right).
\end{align*}

\end{lemma}
\begin{proof}
    Proof follows from the note that for any $x$, $y$ the expression $V(x, z) - V(y, z)$ is linear in terms of $\nabla h(z)$.
\end{proof}

Now we present the theorem to state the convergance of Algorithm \ref{alg:new}.
\begin{theorem}\label{theorem1}
Consider the problem \eqref{eq:VI}+\eqref{eq:sum} under Assumptions~\ref{ass1}, \ref{ass2} and \ref{ass3a}. Let  $\{x_B^k\}$ be the sequence generated by Algorithm~\ref{alg:new} with tuning of $\eta, \theta, \alpha, \beta, \gamma$  as follows:
\begin{align}
    0 < \gamma = p \leq \frac{1}{16}, \label{th1:gamma} 
\end{align} 
\begin{align}
    \quad \eta = \min\left\{\frac{\sqrt{\gamma b}}{8\bar L_2}, \frac{1}{8L_2}\right\}. \label{th1:eta}
\end{align} 
Then for $x_S = \frac{1}{KS} \sum\limits_{s=0}^{S-1}\sum\limits_{k=0}^{K-1} x_B^k$ it holds that
\begin{align}
    \mathbb{E}\left[\operatorname{Gap} (x_S)\right]
    \leq &  \frac{(2+K\gamma)}{\eta K S }\max\limits_{x\in \mathcal{C}}\left\{ V(x,x_0) \right\}\notag.
\end{align}

\end{theorem}
\begin{proof}
We start from using Lemma \ref{lem:1} after plugging parameters $u = \eta\Delta^k, \; x^\dag = x^{k+1}_s, \; x_1 = \overline{\omega}_s, \; x_2 = x^k_s$, and get 
\begin{align*}
      V(x, x^{k+1}_s) 
      \leq &\eta\langle \Delta^k, x - x^{k+1}_s \rangle - \gamma V(x^{k+1}_s, \overline{\omega}_s) + \gamma V(x, \overline\omega_s) 
    \\ & - (1-\gamma)\left(V(x^{k+1}_s, x^k_s) - V(x, x^k_s)\right) + \eta g(x) - \eta g(x^{k+1}_s).
\end{align*}
Then, after applying (\ref{lem2a:Ek}) from Lemma \ref{lem:2a} and simple algebra, we get
\begin{align*}
V(x, x^{k+1}_s) \leq & - \eta \langle\mathbb{E}_k\left[\Delta^k\right], x^{k+1}_s-x\rangle + \eta \langle\mathbb{E}_k\left[\Delta^k\right]-\Delta^k, x^{k+1}_s-x^k_s\rangle
\\&+ \eta \langle\mathbb{E}_k\left[\Delta^k\right]-\Delta^k, x^k_s-x\rangle- \gamma V(x^{k+1}_s, \overline{\omega}_s) + \gamma V(x_s, \overline{\omega}_s)
\\ & - (1-\gamma)(V(x^{k+1}_s, x^k_s) - V(x, x^k_s)) + \eta g(x) - \eta g(x^{k+1}_s)
    \\\leq & - \eta \langle F(x^k_s)+F(x^k_s)-F(x^{k-1}_s), x^{k+1}_s-x\rangle 
\\&+\eta \langle\mathbb{E}_k\left[\Delta^k\right]-\Delta^k, x^{k+1}_s-x^k_s\rangle + \eta \langle\mathbb{E}_k\left[\Delta^k\right]-\Delta^k, x^k_s-x\rangle
\\&- \gamma V(x^{k+1}, \overline{\omega}_s) + \gamma V(x, \overline{\omega}_s)
\\ & - (1-\gamma)(V(x^{k+1}_s, x^k_s) - V(x, x^k_s)) + \eta g(x) - \eta g(x^{k+1}_s)
     \\\leq & - \eta \langle F(x^k_s)-F(x^{k+1}_s)+F(x^k_s)-F(x^{k-1}_s), x^{k+1}_s-x\rangle
\\&- \eta \langle F(x^{k+1}_s), x^{k+1}_s-x \rangle + \eta \langle\mathbb{E}_k\left[\Delta^k\right]-\Delta^k, x^{k+1}_s-x^k_s\rangle
\\&+ \eta \langle\mathbb{E}_k\left[\Delta^k\right]-\Delta^k, x^k_s-x\rangle- \gamma V(x^{k+1}_s,\overline{\omega}_s) + \gamma V(x, \overline{\omega}_s) 
\\ & - (1-\gamma)(V(x^{k+1}_s, x^k_s) - V(x, x^k_s)) + \eta g(x) - \eta g(x^{k+1}_s) .
\end{align*}
By simple rearrangements, we obtain
\begin{align*}
    \eta (g(x^{k+1}_s)&-g(x)) +\eta \langle F(x^{k+1}_s), x^{k+1}_s-x \rangle 
    \\\leq & - \eta \langle F(x^k_s)-F(x^{k+1}_s)+F(x^k_s)-F(x^{k-1}_s), x^{k+1}_s-x\rangle
     \\&
      + \eta \langle\mathbb{E}_k\left[\Delta^k\right]-\Delta^k, x^{k+1}_s-x^k_s\rangle
\\&+ \eta \langle\mathbb{E}_k\left[\Delta^k\right]-\Delta^k, x^k_s-x\rangle- \gamma V(x^{k+1}_s, \overline{\omega}_s) + \gamma V(x, \overline{\omega}_s)
    \\ & - (1-\gamma)(V(x^{k+1}_s, x^k_s) - V(x, x^k_s))  - V(x, x^{k+1}_s) 
    \\\leq & - \eta \langle F(x^k_s)-F(x^{k+1}_s), x^{k+1}_s-x\rangle - \eta\langle F(x^k_s)-F(x^{k-1}_s), x^{k+1}_s-x\rangle
     \\&
      + \eta \langle\mathbb{E}_k\left[\Delta^k\right]-\Delta^k, x^{k+1}_s-x^k_s\rangle
\\&+ \eta \langle\mathbb{E}_k\left[\Delta^k\right]-\Delta^k, x^k_s-x\rangle- \gamma V(x^{k+1}, \overline{\omega}_s) + \gamma V(x, \overline{\omega}_s) 
\\ & - (1-\gamma)(V(x^{k+1}_s, x^k_s) - V(x, x^k_s))  - V(x, x^{k+1}_s)
    \\\leq & - \eta \langle F(x^k)-F(x^{k+1}_s), x^{k+1}_s-x\rangle - \eta\langle F(x^k_s)-F(x^{k-1}_s), x^{k+1}_s-x^k_s\rangle
\\&-\eta\langle F(x^k_s)-F(x^{k-1}_s), x^k_s-x\rangle
      + \eta \langle\mathbb{E}_k\left[\Delta^k\right]-\Delta^k, x^{k+1}_s-x^k_s\rangle
\\&+ \eta \langle\mathbb{E}_k\left[\Delta^k\right]-\Delta^k, x^k_s-x\rangle- \gamma V(x^{k+1}_s, \overline{\omega}_s) + \gamma V(x, \overline{\omega}_s)
    \\ & - (1-\gamma)(V(x^{k+1}_s, x^k_s) - V(x, x^k_s))  - V(x, x^{k+1}_s).
\end{align*}
Next, we apply Lemma \ref{lem:4} and get
\begin{align}
    \eta (g(x^{k+1}_s)&-g(x)) +\eta \langle F(x^{k+1}_s), x^{k+1}_s-x \rangle \notag
    \\\leq & - \eta \langle F(x^k)-F(x^{k+1}_s), x^{k+1}_s-x\rangle-\eta\langle F(x^k_s)-F(x^{k-1}_s), x^{k+1}_s-x^k_s\rangle\notag
    \\& - \eta\langle F(x^k_s)-F(x^{k-1}_s), x^k_s-x\rangle
      + \eta \langle\mathbb{E}_k\left[\Delta^k\right]-\Delta^k, x^{k+1}_s-x^k_s\rangle\notag
\\&+ \eta \langle\mathbb{E}_k\left[\Delta^k\right]-\Delta^k, x^k_s-x\rangle + \frac{ \gamma}{K}\sum_{j=0}^{K-1}\left (V(x, x_{s-1}^j)-V(x_s^{k+1},x_{s-1}^j)\right)\notag
\\ & - (1-\gamma)(V(x^{k+1}_s, x^k_s) - V(x, x^k_s))  - V(x, x^{k+1}_s)\notag
    \\=&- \eta \langle F(x^k)-F(x^{k+1}_s), x^{k+1}_s-x\rangle - \eta\langle F(x^k_s)-F(x^{k-1}_s), x^{k+1}_s-x^k_s\rangle\notag
    \\& - \eta\langle F(x^k_s)-F(x^{k-1}_s), x^k_s-x\rangle
      + \eta \langle\mathbb{E}_k\left[\Delta^k\right]-\Delta^k, x^{k+1}_s-x^k_s\rangle\notag
\\&+ \eta \langle\mathbb{E}_k\left[\Delta^k\right]-\Delta^k, x^k_s-x\rangle + \frac{ \gamma}{K}\sum_{j=0}^{K-1}V(x, x_{s-1}^j)-\frac{ \gamma}{K}\sum_{j=0}^{K-1}V(x_s^{k+1},x_{s-1}^j)\notag
\\ & - (1-\gamma)V(x^{k+1}_s, x^k_s) + (1-\gamma)V(x, x^k_s)  - V(x, x^{k+1}_s).\label{th1:before_Jensen}
\end{align}
Applying the fact (\ref{main_inequality}) to $\|x_s^{k+1}-x_{s-1}^j\|^2$, Jensen's inequality and line \ref{Alg1Line10} of Algorithm \ref{alg:new}, we state
\begin{align}
-\frac{ \gamma}{K}\sum_{j=0}^{K-1}V(x_s^{k+1},x_{s-1}^j) \leq & -\frac{ \gamma}{2K}\sum_{j=0}^{K-1}\|x_s^{k+1}-x_{s-1}^j\|^2  \notag
\\ \leq & -\frac{\gamma}{2}\left\|\frac{1}{K}\sum_{j=0}^{K-1} \left( x_s^{k+1}-x_{s-1}^j \right)\right\|^2\notag
\\ \leq & -\frac{ \gamma}{2}\|x_s^{k+1}-\omega_{s}\|^2. \label{th1:from_Jensen}
\end{align}
Substituting the fact (\ref{th1:from_Jensen}) to (\ref{th1:before_Jensen}), we estimate
\begin{align}
    \eta (g(x^{k+1}_s)-&g(x)) +\eta \langle F(x^{k+1}_s), x^{k+1}_s-x \rangle  \notag
    \\\leq &- \eta \langle F(x^k)-F(x^{k+1}_s), x^{k+1}_s-x\rangle - \eta\langle F(x^k_s)-F(x^{k-1}_s), x^{k+1}_s-x^k_s\rangle \notag
\\& - \eta\langle F(x^k_s)-F(x^{k-1}_s), x^k_s-x\rangle+ \eta \langle\mathbb{E}_k\left[\Delta^k\right]-\Delta^k, x^{k+1}_s-x^k_s\rangle  \notag
\\&+ \eta \langle\mathbb{E}_k\left[\Delta^k\right]-\Delta^k, x^k_s-x\rangle + \frac{ \gamma}{K}\sum_{j=0}^{K-1}V(x, x_{s-1}^j)-\frac{ \gamma}{2}\|x_s^{k+1}-\omega_{s}\|^2 \notag
\\ & - (1-\gamma)V(x^{k+1}_s, x^k_s) + (1-\gamma)V(x, x^k_s)  - V(x, x^{k+1}_s). \label{th1:res1}
\end{align}
For the next step we apply the Young's inequality (\ref{ap:Young}), Assumption \ref{ass3a}, the choice of step (\ref{th1:eta}) and the fact that $\|\cdot \|_2\leq \|\cdot \|$. This allows us to achieve
\begin{align}
    -\eta \langle F(x_s^k)- F(x_s^{k-1}), x_s^{k+1}-x_s^k \rangle \leq & 2\eta^2 \|F(x_s^k)- F(x_s^{k-1})\|^2_2 + \frac{1}{8} \| x_s^{k+1}-x_s^k \|^2_2 \notag 
    \\\leq & \bar{L}^2_2\eta^2 \|x_s^k- x_s^{k-1}\|^2_2 + \frac{1}{8} \| x_s^{k+1}-x_s^k \|^2_2 \notag
    \\ \leq & \frac{1}{32} \|x_s^k- x_s^{k-1}\|^2_2 + \frac{1}{8} \| x_s^{k+1}-x_s^k \|^2_2 \notag
    \\ \leq & \frac{1}{32} \|x_s^k- x_s^{k-1}\|^2 + \frac{1}{8} \| x_s^{k+1}-x_s^k \|^2  . \label{th1:res1:Young}
\end{align}
We substitute (\ref{th1:res1:Young}) to (\ref{th1:res1}) to get 

\begin{align}
    \eta (g(x^{k+1}_s)-&g(x)) + \langle F(x^{k+1}_s), x^{k+1}_s-x \rangle  \notag
    \\\leq &- \eta \langle F(x^k)-F(x^{k+1}_s), x^{k+1}_s-x\rangle -\eta\langle F(x^k_s)-F(x^{k-1}_s), x^k_s-x\rangle \notag
\\&+ \frac{1}{32} \|x_B^k- x_s^{k-1}\|^2 + \frac{1}{4} V(x_s^{k+1},x_B^k) + \eta \langle\mathbb{E}_k\left[\Delta^k\right]-\Delta^k, x^{k+1}_s-x^k_s\rangle  \notag
\\&+ \eta \langle\mathbb{E}_k\left[\Delta^k\right]-\Delta^k, x^k_s-x\rangle + \frac{ \gamma}{K}\sum_{j=0}^{K-1}V(x, x_{s-1}^j)-\frac{ \gamma}{2}\|x_s^{k+1}-\omega_{s}\|^2  \notag
\\ & - (1-\gamma)V(x^{k+1}_s, x^k_s) + (1-\gamma)V(x, x^k_s)  - V(x, x^{k+1}_s). \notag
\end{align}
Next, we sum for all $k$ from $0$ to $K-1$ and obtain
\begin{align}
    \eta \sum_{k=0}^{K-1}\Big[(g(x^{k+1}_s)-&g(x)) +\eta \langle F(x^{k+1}_s), x^{k+1}_s-x \rangle\Big]   \notag
    \\\leq &- \eta \sum_{k=0}^{K-1}\langle F(x^k_s)-F(x^{k+1}_s), x^{k+1}_s-x\rangle\notag
\\&- \eta \sum_{k=0}^{K-1}\langle F(x^k_s)-F(x^{k-1}_s), x^k_s-x\rangle \notag
\\&+ \frac{1}{32} \sum_{k=0}^{K-1}\|x_s^k- x_s^{k-1}\|^2 + \frac{1}{4}\sum_{k=0}^{K-1} V(x_s^{k+1},x_s^k)\notag
\\&+ \eta \sum_{k=0}^{K-1}\langle\mathbb{E}_k\left[\Delta^k\right]-\Delta^k, x^{k+1}_s-x^k_s\rangle  \notag
\\&+ \eta \sum_{k=0}^{K-1}\langle\mathbb{E}_k\left[\Delta^k\right]-\Delta^k, x^k_s-x\rangle  + \frac{ \gamma}{K}\sum_{k=0}^{K-1}\sum_{j=0}^{K-1}V(x, x_{s-1}^j)\notag
\\&-\frac{ \gamma}{2}\sum_{k=0}^{K-1}\|x_s^{k+1}-\omega_{s}\|^2 - (1-\gamma)\sum_{k=0}^{K-1}V(x^{k+1}_s, x^k_s) \notag
\\ &  +  (1-\gamma)\sum_{k=0}^{K-1}V(x, x^k_s)  - \sum_{k=0}^{K-1}V(x, x^{k+1}_s) \notag
    \\\leq &\eta \langle F(x^{-1}_s)-F(x^0_s), x_s^0 - x \rangle - \eta\langle F(x^{K-1}_s)-F(x^K_s), x^K_s-x\rangle \notag
    \\ & + \frac{1}{32} \sum_{k=0}^{K-1}\|x_s^k- x_s^{k-1}\|^2 + \frac{1}{4} \sum_{k=0}^{K-1} V(x_s^{k+1},x_s^k)  \notag
\\&+ \eta \sum_{k=0}^{K-1}\langle\mathbb{E}_k\left[\Delta^k\right]-\Delta^k, x^{k+1}_s-x^k_s\rangle \notag
\\& + \eta \sum_{k=0}^{K-1}\langle\mathbb{E}_k\left[\Delta^k\right]-\Delta^k, x^k_s-x\rangle \notag
\\& + \gamma \sum_{k=0}^{K-1}V(x, x_{s-1}^k) -\frac{ \gamma}{2}\sum_{k=0}^{K-1}\|x_s^{k+1}-\omega_{s}\|^2\notag
\\& - (1-\gamma)\sum_{k=0}^{K-1}V(x^{k+1}_s, x^k_s)\notag
  +  (1-\gamma)\sum_{k=0}^{K-1}V(x, x^k_s)\notag
\\&  - \sum_{k=0}^{K-1}V(x, x^{k+1}_s)
. \notag
\end{align}
We take sum for all $s$ from 0 to $S-1$ and apply lines \ref{Alg1Line12} and \ref{Alg1Line14} of Algorithm \ref{alg:new} to achieve
\begin{align}
    \eta \sum_{s=0}^{S-1}\sum_{k=0}^{K-1}\Big[(g(x^{k+1}_s)&-g(x)) +\eta \langle F(x^{k+1}_s), x^{k+1}_s-x \rangle\Big]  \notag
    \\\leq &\eta \sum_{s=0}^{S-1} \langle F(x^{-1}_s)-F(x^0_s), x_s^0 - x \rangle \notag
\\ & - \eta \sum_{s=0}^{S-1} \langle F(x^{K-1}_s)-F(x^K_s), x^K_s-x\rangle \notag
\\ & + \frac{1}{32} \sum_{s=0}^{S-1} \sum_{k=0}^{K-1}\|x_s^k- x_s^{k-1}\|^2 + \frac{1}{4} \sum_{s=0}^{S-1} \sum_{k=0}^{K-1} V(x_s^{k+1},x_s^k)  \notag
\\&+ \eta \sum_{s=0}^{S-1} \sum_{k=0}^{K-1}\langle\mathbb{E}_k\left[\Delta^k\right]-\Delta^k, x^{k+1}_s-x^k_s\rangle \notag
\\& + \eta \sum_{s=0}^{S-1} \sum_{k=0}^{K-1}\langle\mathbb{E}_k\left[\Delta^k\right]-\Delta^k, x^k_s-x\rangle \notag
\\& + \gamma \sum_{s=0}^{S-1} \sum_{k=0}^{K-1}V(x, x_{s-1}^k) -\frac{ \gamma}{2}\sum_{s=0}^{S-1} \sum_{k=0}^{K-1}\|x_s^{k+1}-\omega_{s}\|^2\notag
\\&- (1-\gamma)\sum_{s=0}^{S-1} \sum_{k=0}^{K-1}V(x^{k+1}_s, x^k_s)  +  (1-\gamma)\sum_{s=0}^{S-1} \sum_{k=0}^{K-1}V(x, x^k_s) \notag
\\&- \sum_{s=0}^{S-1} \sum_{k=0}^{K-1}V(x, x^{k+1}_s) \notag
    \\ = &\eta \sum_{s=0}^{S-1} \langle F(x^{-1}_s)-F(x^0_s), x_s^0 - x \rangle \notag
\\& - \eta \sum_{s=0}^{S-1} \langle F(x^{-1}_{s+1})-F(x^0_{s+1}), x^0_{s+1}-x\rangle \notag
\\ & + \frac{1}{32} \sum_{s=0}^{S-1} \sum_{k=0}^{K-1}\|x_B^k- x_s^{k-1}\|^2 + \frac{1}{4} \sum_{s=0}^{S-1} \sum_{k=0}^{K-1} V(x_s^{k+1},x_B^k)  \notag
\\&+ \eta \sum_{s=0}^{S-1} \sum_{k=0}^{K-1}\langle\mathbb{E}_k\left[\Delta^k\right]-\Delta^k, x^{k+1}_s-x^k_s\rangle \notag
\\& + \eta \sum_{s=0}^{S-1} \sum_{k=0}^{K-1}\langle\mathbb{E}_k\left[\Delta^k\right]-\Delta^k, x^k_s-x\rangle \notag
\\& + \gamma \sum_{s=0}^{S-1} \sum_{k=0}^{K-1}V(x, x_{s-1}^k) -\frac{ \gamma}{2}\sum_{s=0}^{S-1} \sum_{k=0}^{K-1}\|x_s^{k+1}-\omega_{s}\|^2\notag
\\&- (1-\gamma)\sum_{s=0}^{S-1} \sum_{k=0}^{K-1}V(x^{k+1}_s, x^k_s)  +  (1-\gamma)\sum_{s=0}^{S-1} \sum_{k=0}^{K-1}V(x, x^k_s) \notag
\\&- \sum_{s=0}^{S-1} \sum_{k=0}^{K-1}V(x, x^{k+1}_s). \label{th1:res2:sum}
\end{align}
Simple algebra and line \ref{Alg1Line2} from Algorithm \ref{alg:new} allows us to get 
\begin{align}
    \eta \sum_{s=0}^{S-1} \langle F(x^{-1}_s)&-F(x^0_s), x_s^0 - x \rangle - \eta \sum_{s=0}^{S-1} \langle F(x^{-1}_{s+1})-F(x^0_{s+1}), x^0_{s+1}-x\rangle \notag
    \\ = & \eta \langle F(x_0^S)-F(x_S^{-1}), x_S^0-x \rangle -\eta\langle F(x_0^0)-F(x_0^{-1}), x- x_0^0\rangle \notag
    \\ = & \eta \langle F(x_0^S)-F(x_S^{-1}), x_S^0-x \rangle\label{th1:res2:simple_algebra}
\end{align}
Applying the Young's inequality (\ref{ap:Young}), Assumption \ref{ass3a}, the choice of step (\ref{th1:eta}) and the fact that $\|\cdot \|_2\leq \|\cdot \|$, we get 
\begin{align}
     \eta \langle F(x^{0}_S)-F(x^{-1}_S), x^0_S-x\rangle \leq & \frac{\eta^2}{2}\|F(x_S^0)- F(x_S^{-1})\|_2^2 + \frac{1}{2}\|x_S^0 - x\|^2_2 \notag
     \\\leq & \frac{\eta^2 L^2_2}{2}\|x_S^{0}-x_S^{-1}\|^2_2+\frac{1}{2}\|x_S^0 - x\|^2_2\notag
     \\ \leq &\frac{1}{128}\|x_S^{0}-x_S^{-1}\|^2_2+\frac{1}{2}\|x_S^0- x\|^2_2\notag
     \\ \leq &\frac{1}{128}\|x_S^{0}-x_S^{-1}\|^2+\frac{1}{2}\|x_S^0- x\|^2
     .\label{th1:res2:Young}
\end{align}
Inequalities (\ref{th1:res2:Young}) and (\ref{th1:res2:simple_algebra}) allow us to turn (\ref{th1:res2:sum}) into
\begin{align}
    \eta \sum_{s=0}^{S-1}\sum_{k=0}^{K-1}\Big[(g(x^{k+1}_s)&-g(x)) +\eta \langle F(x^{k+1}_s), x^{k+1}_s-x \rangle\Big] \notag 
    \\\leq & \frac{1}{128}\|x_{S}^{0}-x_{S}^{-1}\|^2+ \frac{1}{2}\|x_S^0 - x\|^2 + \frac{1}{4} \sum_{s=0}^{S-1}\sum_{k=0}^{K-1} V(x_s^{k+1},x_s^k)\notag
\\& + \frac{1}{32} \sum_{s=0}^{S-1} \sum_{k=0}^{K-1}\|x_B^k- x_s^{k-1}\|^2\notag
\\& + \eta \sum_{s=0}^{S-1}\sum_{k=0}^{K-1}\langle\mathbb{E}_k\left[\Delta^k\right]-\Delta^k, x^{k+1}_s-x^k_s\rangle  \notag
\\&+ \eta \sum_{s=0}^{S-1}\sum_{k=0}^{K-1}\langle\mathbb{E}_k\left[\Delta^k\right]-\Delta^k, x^k_s-x\rangle + \gamma\sum_{s=1}^{S-1}\sum_{k=0}^{K-1}V(x, x_{s-1}^k) \notag
\\&-\frac{ \gamma}{2}\sum_{s=0}^{S-1}\sum_{k=0}^{K-1}\|x_s^{k+1}-\omega_{s}\|^2 - (1-\gamma)\sum_{s=0}^{S-1}\sum_{k=0}^{K-1}V(x^{k+1}_s, x^k_s) \notag
\\& + (1-\gamma)\sum_{s=0}^{S-1}\sum_{k=0}^{K-1}V(x, x^k_s) - \sum_{s=0}^{S-1}\sum_{k=0}^{K-1}V(x, x^{k+1}_s) \notag
    \\\leq & \frac{1}{128}\|x_{S-1}^{K}-x_{S-1}^{K-1}\|^2+ V(x, x_S^0) + \frac{1}{4} \sum_{s=0}^{S-1}\sum_{k=0}^{K-1} V(x_s^{k+1},x_s^k)\notag
\\& + \frac{1}{32} \sum_{s=0}^{S-1} \sum_{k=0}^{K-1}\|x_B^k- x_s^{k-1}\|^2\notag
\\& + \eta \sum_{s=0}^{S-1}\sum_{k=0}^{K-1}\langle\mathbb{E}_k\left[\Delta^k\right]-\Delta^k, x^{k+1}_s-x^k_s\rangle  \notag
\\&+ \eta \sum_{s=0}^{S-1}\sum_{k=0}^{K-1}\langle\mathbb{E}_k\left[\Delta^k\right]-\Delta^k, x^k_s-x\rangle + \gamma\sum_{s=0}^{S-1}\sum_{k=0}^{K-1}V(x, x_{s-1}^k) \notag
\\&-\frac{ \gamma}{2}\sum_{s=0}^{S-1}\sum_{k=0}^{K-1}\|x_s^{k+1}-\omega_{s}\|^2 - (1-\gamma)\sum_{s=0}^{S-1}\sum_{k=0}^{K-1}V(x^{k+1}_s, x^k_s) \notag
\\& + (1-\gamma)\sum_{s=0}^{S-1}\sum_{k=0}^{K-1}V(x, x^k_s) - \sum_{s=0}^{S-1}\sum_{k=0}^{K-1}V(x, x^{k+1}_s) . \label{th1:res3}
\end{align}
In this step, we again used lines \ref{Alg1Line12} and \ref{Alg1Line14} form Algorithm \ref{alg:new} and the fact (\ref{main_inequality}).
This result and line \ref{Alg1Line12} of Algorithm \ref{alg:new} gives us
\begin{align}
    \gamma\sum_{s=0}^{S-1}\sum_{k=0}^{K-1}V(x, &x_{s-1}^k) + (1-\gamma)\sum_{s=0}^{S-1}\sum_{k=0}^{K-1}V(x, x^k_s) - \sum_{s=0}^{S-1}\sum_{k=0}^{K-1}V(x, x^{k+1}_s)+ V(x, x_S^0) \notag
    \\  \leq &\ \gamma\sum_{k=0}^{K-1}V(x, x^k_{-1}) +\gamma\left(\sum_{s=0}^{S-2}\sum_{k=0}^{K-1}V(x, x_s^k)- \sum_{s=0}^{S-1}\sum_{k=0}^{K-1}V(x, x^k_s)\right)\notag
\\ & + \Bigg( \sum_{s=0}^{S-1}\sum_{k=0}^{K-1}V(x, x^k_s)- \sum_{s=0}^{S-1}\sum_{k=1}^{K}V(x, x^{k}_s) \Bigg) + V(x, x_S^0) \notag
    \\ \leq &\gamma\sum_{k=0}^{K-1}V(x, x^k_{-1}) -\gamma \sum_{k=0}^{K-1} V(x,x_{S-1}^k) + \sum_{s=0}^{S-1} V(x, x_s^0)\notag
\\ &- \sum_{s=0}^{S-1} V(x,x^K_s)+ V(x, x_S^0)\notag
    \\ = &\gamma K V(x, x^0_{0}) -\gamma \sum_{k=0}^{K-1} V(x,x_{S-1}^k)\notag
\\ & + \sum_{s=0}^{S-1} V(x, x_s^0)- \sum_{s=1}^{S} V(x,x^0_s) + V(x, x_S^0)\notag
    \\= &\gamma K V(x, x^0_{0}) - \gamma \sum_{k=0}^{K-1} V(x,x_{S-1}^k) + V(x, x_0^0) - V(x, x_S^0)+ V(x, x_S^0) \notag
\\\leq & (1 + \gamma K)V(x,x_0^0).\notag
\end{align}
We apply this result to $(\ref{th1:res3})$ and get
\begin{align}
    \eta \sum_{s=0}^{S-1}\sum_{k=0}^{K-1}\Big[(g(x^{k+1}_s)&-g(x)) +\eta \langle F(x^{k+1}_s), x^{k+1}_s-x \rangle\Big]  \notag
    \\ \leq &\frac{1}{128}\|x_{S-1}^{K}-x_{S-1}^{K-1}\|^2+\frac{1}{32} \sum_{s=0}^{S-1} \sum_{k=0}^{K-1}\|x_s^k- x_s^{k-1}\|^2 \notag
\\& + \frac{1}{4}  \sum_{s=0}^{S-1}\sum_{k=0}^{K-1} V(x_s^{k+1},x_s^k)\notag
\\ & +  \eta \sum_{s=0}^{S-1}\sum_{k=0}^{K-1}\Big[ \langle\mathbb{E}_k\left[\Delta^k\right]-\Delta^k, x^{k+1}_s-x^k_s\rangle\big] \notag
\\&+ \eta \sum_{s=0}^{S-1} \sum_{k=0}^{K-1}\Big[\langle\mathbb{E}_k\left[\Delta^k\right]-\Delta^k, x_s^k-x\rangle\Big]\notag
\\ & -\frac{ \gamma}{2}\sum_{s=0}^{S-1}\sum_{k=0}^{K-1}\|x_s^{k+1}-\omega_{s}\|^2- (1-\gamma)\sum_{s=0}^{S-1}\sum_{k=0}^{K-1}V(x^{k+1}_s, x^k_s)\notag
\\ & +(1 + \gamma K)V(x,x_0^0). \notag 
\end{align}
After that we take maximum and expectation of both sides, we used fact that the maximum of sum is less then the sum of maximums. In addition, here we use property of conditional expectation to obtain $\mathbb{E}\Bigg[ \sum_{s=0}^{S-1}\sum_{k=0}^{K-1}\Big[ \langle\mathbb{E}_k\left[\Delta^k\right]-\Delta^k, x^k_s\rangle\big] \Bigg] = 0.$ Thus,

\begin{align}
\mathbb{E}\Bigg[ \max\limits_{x\in \mathcal{C}} \Bigg\{\sum_{s=0}^{S-1}\sum_{k=0}^{K-1}&\Big[\eta (g(x^{k+1}_s)-g(x)) + \eta \langle F(x^{k+1}_s), x^{k+1}_s-x \rangle \Big] \Bigg\}\Bigg]\notag
    \\ 
    \leq &
    \frac{1}{128}\mathbb{E}\Bigg[\|x_{S-1}^{K}-x_{S-1}^{K-1}\|^2\Big]  +\frac{1}{32} \mathbb{E}\Bigg[ \sum_{s=0}^{S-1} \sum_{k=0}^{K-1}\|x_s^k- x_s^{k-1}\|^2 \Bigg]\notag
\\
&+ \left(\gamma-\frac{3}{4}\right) \mathbb{E}\Bigg[ \sum_{s=0}^{S-1}\sum_{k=0}^{K-1} V(x_s^{k+1},x_s^k) \Bigg]\notag
\\
& +  \eta \mathbb{E}\Bigg[ \sum_{s=0}^{S-1}\sum_{k=0}^{K-1}\Big[ \langle\mathbb{E}_k\left[\Delta^k\right]-\Delta^k, x^{k+1}_s-x^k_s\rangle\big] \Bigg]\notag
\\
&+ \eta \mathbb{E}\Bigg[ \max\limits_{x\in \mathcal{C}} \left\{ \sum_{s=0}^{S-1} \sum_{k=0}^{K-1}\Big[\langle\mathbb{E}_k\left[\Delta^k\right]-\Delta^k, x_s^k - x\rangle\Big] \right\} \Bigg] \notag
\\& - \frac{ \gamma}{2}\mathbb{E}\Bigg[ \sum_{s=0}^{S-1}\sum_{k=0}^{K-1}\|x_s^{k+1}-\omega_{s}\|^2 \Bigg]+\max\limits_{x\in \mathcal{C}} \Phi_0(x)\notag
    \\ 
    \leq &
    \frac{1}{128}\mathbb{E}\Big[\|x_{S-1}^{K}-x_{S-1}^{K-1}\|^2\Big] + \frac{1}{32} \mathbb{E}\Bigg[ \sum_{s=0}^{S-1} \sum_{k=0}^{K-1}\|x_s^k- x_s^{k-1}\|^2 \Bigg]\notag
\\
&+ \left(\gamma-\frac{3}{4}\right) \mathbb{E}\Bigg[ \sum_{s=0}^{S-1}\sum_{k=0}^{K-1} V(x_s^{k+1},x_s^k) \Bigg]\notag
\\
& +  \eta \mathbb{E}\Bigg[ \sum_{s=0}^{S-1}\sum_{k=0}^{K-1}\Big[ \langle\mathbb{E}_k\left[\Delta^k\right]-\Delta^k, x^{k+1}_s-x^k_s\rangle\big] \Bigg]\notag
\\
&+ \eta \mathbb{E}\Bigg[ \max\limits_{x\in \mathcal{C}} \left\{ \sum_{s=0}^{S-1} \sum_{k=0}^{K-1}\Big[\langle\mathbb{E}_k\left[\Delta^k\right]-\Delta^k, x\rangle\Big] \right\} \Bigg] \notag
\\& - \frac{ \gamma}{2}\mathbb{E}\Bigg[ \sum_{s=0}^{S-1}\sum_{k=0}^{K-1}\|x_s^{k+1}-\omega_{s}\|^2 \Bigg]+(1 + \gamma K)\max\limits_{x\in \mathcal{C}} V(x,x_0^0). \label{th1:res4}
\end{align}
Using the Young's inequality (\ref{ap:Young}), we can estimate
\begin{align}
\mathbb{E}\left[ \eta\langle\mathbb{E}_k[\Delta^k] - \Delta^k, x_s^{k+1} - x_s^k\rangle \right] 
    \leq 2\eta^2\mathbb{E}\left[\|\mathbb{E}_k[\Delta^k]-\Delta^k\|^2_2  \right] 
+ \frac{1}{8}\mathbb{E}\left[\|x_s^{k+1} - x_s^k\|^2_2\right]. \notag
\end{align}
Then
\begin{align}
    \eta \mathbb{E}\Bigg[ \sum_{s=0}^{S-1}\sum_{k=0}^{K-1}&\Big[ \langle\mathbb{E}_k \left[\Delta^k\right]-\Delta^k, x^{k+1}_s-x^k_s\rangle\big] \Bigg]\notag
    \\
    \leq&2\eta^2  \sum_{s=0}^{S-1}\sum_{k=0}^{K-1} \mathbb{E} \|\mathbb{E}_k\left[\Delta^k\right]-\Delta^k \|^2_2
+\frac{1}{8}\sum_{s=0}^{S-1}\sum_{k=0}^{K-1} \Bigg[ \mathbb{E}\left[\|x_s^{k+1} - x_s^k\|^2_2\right] \Bigg].
\label{th1:res4:est}
\end{align}
$\mathbb{E}_k\left[\Delta^k\right]-\Delta^k$ is a stochastic porcess with $\mathbb{E}[\mathbb{E}_k\left[\Delta^k\right]-\Delta^k]=0$. Therefore, according to Lemma \ref{lem:3a}
\begin{align}
    \eta \mathbb{E}\Bigg[\max\limits_{x\in \mathcal{C}} \Bigg\{ \sum_{s=0}^{S-1}\sum_{k=0}^{K-1} \langle\mathbb{E}_k &\left[\Delta^k\right] -\Delta^k, x\rangle \Bigg\} \Bigg]\notag
    \\ \leq &
    \frac{1}{2} \max\limits_{x\in \mathcal{C}} \|x - x^0_0\|^2_2
+ \frac{\eta^2}{2}  \sum_{s=0}^{S-1}\sum_{k=0}^{K-1} \mathbb{E} \|\mathbb{E}_k\left[\Delta^k\right]-\Delta^k \|^2_2.\label{th1:res4:lem3}
\end{align}
Applying (\ref{th1:res4:est}) and (\ref{th1:res4:lem3}) to (\ref{th1:res4}), we get
\begin{align}
\mathbb{E}\Bigg[ \max\limits_{x\in \mathcal{C}} \Bigg\{\sum_{s=0}^{S-1}\sum_{k=0}^{K-1}\Big[&\eta (g(x^{k+1}_s)-g(x)) + \eta \langle F(x^{k+1}_s), x^{k+1}_s-x \rangle \Big] \Bigg\}\Bigg]\notag
    \\
    \leq&
    \frac{1}{128}\mathbb{E}\Big[\|x_{S-1}^{K}-x_{S-1}^{K-1}\|^2\Big]+\frac{1}{32} \mathbb{E}\Bigg[ \sum_{s=0}^{S-1} \sum_{k=0}^{K-1}\|x_s^k- x_s^{k-1}\|^2 \Bigg]\notag
\\
&+ \left(\gamma-\frac{3}{4}\right) \mathbb{E}\Bigg[ \sum_{s=0}^{S-1}\sum_{k=0}^{K-1} V(x_s^{k+1},x_s^k) \Bigg]\notag
\\
&+2\eta^2  \sum_{s=0}^{S-1}\sum_{k=0}^{K-1} \mathbb{E} \|\mathbb{E}_k\left[\Delta^k\right]-\Delta^k \|^2_2\notag
\\&+\frac{1}{8}\sum_{s=0}^{S-1}\sum_{k=0}^{K-1} \Bigg[ \mathbb{E}\left[\|x_s^{k+1} - x_B^k\|^2_2\right] \Bigg]\notag
\\
&+\frac{1}{2} \max\limits_{x\in \mathcal{C}} \|x - x^0_0\|^2_2 + \frac{\eta^2}{2}  \sum_{s=0}^{S-1}\sum_{k=0}^{K-1} \mathbb{E} \|\mathbb{E}_k\left[\Delta^k\right]-\Delta^k \|^2_2\notag
\\
& - \frac{ \gamma}{2}\mathbb{E}\Bigg[ \sum_{s=0}^{S-1}\sum_{k=0}^{K-1}\|x_s^{k+1}-\omega_{s}\|^2 \Bigg]+(1 + \gamma K)\max\limits_{x\in \mathcal{C}} V(x,x_0^0) \notag
    \\
    \leq&
    \frac{1}{128}\mathbb{E}\Big[\|x_{S-1}^{K}-x_{S-1}^{K-1}\|^2\Big] +\frac{1}{32} \mathbb{E}\Bigg[ \sum_{s=0}^{S-1} \sum_{k=0}^{K-1}\|x_s^k- x_s^{k-1}\|^2 \Bigg]\notag
\\
&+ \left(\gamma-\frac{3}{4}\right) \mathbb{E}\Bigg[ \sum_{s=0}^{S-1}\sum_{k=0}^{K-1} V(x_s^{k+1},x_B^k) \Bigg]\notag
\\ &+\frac{5\eta^2}{2}  \sum_{s=0}^{S-1}\sum_{k=0}^{K-1} \mathbb{E} \|\mathbb{E}_k\left[\Delta^k\right]-\Delta^k \|^2_2\notag
\\&
+\frac{1}{8}\sum_{s=0}^{S-1}\sum_{k=0}^{K-1} \Bigg[ \mathbb{E}\left[\|x_s^{k+1} - x_B^k\|^2_2\right] \Bigg]+ \frac{1}{2} \max\limits_{x\in \mathcal{C}} \|x - x^0_0\|^2_2 \notag
\\&
- \frac{ \gamma}{2}\mathbb{E}\Bigg[ \sum_{s=0}^{S-1}\sum_{k=0}^{K-1}\|x_s^{k+1}-\omega_{s}\|^2 \Bigg]+(1 + \gamma K)\max\limits_{x\in \mathcal{C}} V(x,x_0^0). \label{th1:before_CSineq}
\end{align}
The Cauchy–Schwarz(\ref{ap:Cauchy–Schwarz}) inequality allows us to estimate 
\begin{align}
- \frac{ \gamma}{2}\mathbb{E}\Bigg[ \sum_{s=0}^{S-1}&\sum_{k=0}^{K-1}\|x_s^{k+1}-\omega_{s}\|^2 \Bigg] \notag
\\&\leq - \frac{ \gamma}{4}\mathbb{E}\Bigg[ \sum_{s=0}^{S-1}\sum_{k=0}^{K-1}\|x_s^k-\omega_{s}\|^2 \Bigg] + \frac{ \gamma}{2} \mathbb{E}\Bigg[ \sum_{s=0}^{S-1}\sum_{k=0}^{K-1}\|x_s^{k+1}-x_s^k\|^2 \Bigg]\label{th1:CSineq}
\end{align}
Applying Lemma \ref{lem:2a}, the fact (\ref{main_inequality}), property $\|\cdot \|_2\leq \|\cdot \|$ and (\ref{th1:CSineq}) to (\ref{th1:before_CSineq}), we achieve
\begin{align}
\mathbb{E}\Bigg[ \max\limits_{x\in \mathcal{C}} \Bigg\{&\sum_{s=0}^{S-1}\sum_{k=0}^{K-1}\Big[\eta (g(x^{k+1}_s)-g(x)) + \eta \langle F(x^{k+1}_s), x^{k+1}_s-x \rangle \Big] \Bigg\}\Bigg]\notag
    \\
    \leq &
    \frac{1}{128}\mathbb{E}\Bigg[\|x_{S-1}^{K}-x_{S-1}^{K-1}\|^2\Big] + \frac{1}{32} \mathbb{E}\Bigg[ \sum_{s=0}^{S-1} \sum_{k=0}^{K-1}\|x_s^k- x_s^{k-1}\|^2 \Bigg]\notag
\\
&+ \left(\gamma-\frac{3}{4}\right) \mathbb{E}\Bigg[ \sum_{s=0}^{S-1}\sum_{k=0}^{K-1} V(x_s^{k+1},x_s^k) \Bigg]\notag
\\
&+\frac{5\eta^2}{2}   \sum_{s=0}^{S-1}\sum_{k=0}^{K-1} \Bigg[{\frac{2{\overline{{L}}}_2^{2}}{b}}\mathbb{E}\left[\|x^{k}_s-w_{s}\|^2+\|x^{k}_s - x^{k-1}_s\|^2 \right]\Bigg]\notag
\\
&+\frac{1}{8}\sum_{s=0}^{S-1}\sum_{k=0}^{K-1} \Bigg[ \mathbb{E}\left[\|x_s^{k+1} - x_s^k\|^2_2\right] \Bigg]+ \frac{1}{2} \max\limits_{x\in \mathcal{C}} \|x - x^0_0\|^2_2  \notag
\\
& - \frac{ \gamma}{2}\mathbb{E}\Bigg[ \sum_{s=0}^{S-1}\sum_{k=0}^{K-1}\|x_s^{k+1}-\omega_{s}\|^2 \Bigg]+(1 + \gamma K)\max\limits_{x\in \mathcal{C}} V(x,x_0^0)\notag
    \\
    \leq &
    \frac{1}{64}\mathbb{E} \left[ V(x_{S-1}^{K},x_{S-1}^{K-1})\right] + \frac{1}{16} \mathbb{E}\Bigg[ \sum_{s=0}^{S-1} \sum_{k=0}^{K-1}V(x_s^k, x_s^{k-1}) \Bigg]\notag
\\
&+ \left(\gamma-\frac{3}{4}\right) \mathbb{E}\Bigg[ \sum_{s=0}^{S-1}\sum_{k=0}^{K-1} V(x_s^{k+1},x_s^k) \Bigg] +\frac{1}{4}\sum_{s=0}^{S-1}\sum_{k=0}^{K-1} \Bigg[ \mathbb{E}\left[V(x_s^{k+1}, x_s^k)\right] \Bigg]\notag
\\
&+\frac{5\eta^2}{2}   \sum_{s=0}^{S-1}\sum_{k=0}^{K-1} \Bigg[{\frac{{\overline{{L}}}_2^{2}}{b}}\mathbb{E}\left[2\|x^{k}_s-w_{s}\|^2+V\left(x^{k}_s,x^{k-1}_s\right)\right]\Bigg]+ \max\limits_{x\in \mathcal{C}} V(x,x_0^0)\notag
\\
&- \frac{ \gamma}{4}\mathbb{E}\Bigg[ \sum_{s=0}^{S-1}\sum_{k=0}^{K-1}\|x_s^k-\omega_{s}\|^2 \Bigg] + \frac{ \gamma}{2} \mathbb{E}\Bigg[ \sum_{s=0}^{S-1}\sum_{k=0}^{K-1}\|x_s^{k+1}-x_s^k\|^2 \Bigg]\notag
\\
&+(1 + \gamma K)\max\limits_{x\in \mathcal{C}} V(x,x_0^0)\notag
    \\
    \leq&
    \frac{1}{64}\mathbb{E} \left[ V(x_{S-1}^{K},x_{S-1}^{K-1})\right]+\left(\frac{1}{16}+ \frac{{5\overline{{L}}_2^2\eta^2}}{2b}\right) \mathbb{E}\Bigg[ \sum_{s=0}^{S-1} \sum_{k=0}^{K-1}V\left(x_s^k, x_s^{k-1}\right)^2 \Bigg]\notag
\\
&+ \left(\gamma-\frac{1}{2}\right) \mathbb{E}\Bigg[ \sum_{s=0}^{S-1}\sum_{k=0}^{K-1} V(x_s^{k+1},x_s^k) \Bigg]\notag
\\
&+\frac{5{\overline{{L}}_2^2\eta^2}}{b}  \sum_{s=0}^{S-1}\sum_{k=0}^{K-1} \Bigg[\mathbb{E}\|x_s^k- w_s\|^2\Bigg]+\max\limits_{x\in \mathcal{C}} V(x,x_0^0)\notag
\\
&- \frac{ \gamma}{4}\mathbb{E}\Bigg[ \sum_{s=0}^{S-1}\sum_{k=0}^{K-1}\|x_s^k-\omega_{s}\|^2 \Bigg] + \gamma \mathbb{E}\Bigg[ \sum_{s=0}^{S-1}\sum_{k=0}^{K-1}V(x_s^{k+1},x_s^k) \Bigg]\notag
\\
&+(1 + \gamma K)\max\limits_{x\in \mathcal{C}} V(x,x_0^0). \label{th1:res6}
\end{align}
Next, we apply the fact (\ref{main_inequality}), 
inequality (\ref{th1:gamma}), lines \ref{Alg1Line12} and \ref{Alg1Line14} of Algorithm \ref{alg:new}, the choice of the step (\ref{th1:eta}) and the initialization (line \ref{Alg1Line2}) to get
\begin{align}
\frac{1}{64}\mathbb{E} & \left[ V(x_{S-1}^{K}, x_{S-1}^{K-1})\right]+\left(\frac{1}{16}+ \frac{{5\overline{{L}}_2^2\eta^2}}{2b}\right) \mathbb{E}\Bigg[ \sum_{s=0}^{S-1} \sum_{k=0}^{K-1}V\left(x_s^k, x_s^{k-1}\right) \Bigg]\notag
\\&+ \left(\gamma-\frac{1}{2}\right) \mathbb{E}\Bigg[ \sum_{s=0}^{S-1}\sum_{k=0}^{K-1} V(x_s^{k+1},x_s^k) \Bigg]+\gamma \mathbb{E}\Bigg[ \sum_{s=0}^{S-1}\sum_{k=0}^{K-1}V(x_s^{k+1},x_s^k) \Bigg]\notag
    \\ \leq & \frac{1}{64}\mathbb{E} \left[ V(x_{S-1}^{K},x_{S-1}^{K-1})\right] + \left(\frac{1}{16}+ \frac{{5\overline{{L}}_2^2\eta^2}}{2b}\right) \mathbb{E}\Bigg[ \sum_{s=0}^{S-1} \sum_{k=1}^{K-1}V\left(x_s^k, x_s^{k-1}\right) \Bigg] \notag
\\ &+ \left(\frac{1}{16}+ \frac{{5\overline{{L}}_2^2\eta^2}}{2b}\right) \mathbb{E}\Bigg[ \sum_{s=0}^{S-1} V\left(x_s^0, x_s^{-1}\right) \Bigg] \notag
\\ &+ \left(2\gamma-\frac{1}{2}\right) \mathbb{E}\Bigg[ \sum_{s=0}^{S-1}\sum_{k=1}^{K-1} V(x_s^k,x_s^{k-1}) \Bigg] + \left(\gamma-\frac{1}{2}\right) \mathbb{E}\Bigg[ \sum_{s=0}^{S-1} V(x_s^k,x_s^{K-1}) \Bigg]\notag
    \\ \leq & \left(\frac{1}{16}+ \frac{{5\overline{{L}}_2^2\eta^2}}{2b} + 2\gamma-\frac{1}{2} \right)\mathbb{E}\Bigg[ \sum_{s=0}^{S-1} \sum_{k=1}^{K-1}V\left(x_s^k, x_s^{k-1}\right) \Bigg]
      \notag
\\ & + \frac{1}{64}\mathbb{E} \left[ V(x_{S-1}^{K},x_{S-1}^{K-1})\right] + \left(\frac{1}{16}+ \frac{{5\overline{{L}}_2^2\eta^2}}{2b}\right) \mathbb{E}\Bigg[ \sum_{s=0}^{S-1} V\left(x_s^0, x_s^{-1}\right) \Bigg] \notag
\\ & + \left(2\gamma-\frac{1}{2}\right) \mathbb{E}\Bigg[ \sum_{s=0}^{S-1} V(x_s^k,x_s^{K-1}) \Bigg]\notag
    \\ \leq & \left(\frac{1}{16}+ \frac{{5\overline{{L}}_2^2\eta^2}}{2b} + 2\gamma-\frac{1}{2} \right)\mathbb{E}\Bigg[ \sum_{s=0}^{S-1} \sum_{k=1}^{K-1}V\left(x_s^k, x_s^{k-1}\right) \Bigg]
      \notag
\\ & + \frac{1}{64}\mathbb{E} \left[ V(x_{S}^{0},x_{S}^{-1})\right] + \left(\frac{1}{16}+ \frac{{5\overline{{L}}_2^2\eta^2}}{2b}\right) \mathbb{E}\Bigg[ \sum_{s=0}^{S-1} V\left(x_s^0, x_s^{-1}\right) \Bigg] \notag
\\ &+ \left(2\gamma-\frac{1}{2}\right) \mathbb{E}\Bigg[ \sum_{s=1}^{S} V(x_s^0,x_s^{-1}) \Bigg]\notag
    \\ \leq & \left(\frac{1}{16}+ \frac{{5\overline{{L}}_2^2\eta^2}}{2b} + 2\gamma-\frac{1}{2} \right)\mathbb{E}\Bigg[ \sum_{s=0}^{S-1} \sum_{k=1}^{K-1}V\left(x_s^k, x_s^{k-1}\right) \Bigg]
      \notag
\\ & + \left(\frac{1}{16}+ \frac{{5\overline{{L}}_2^2\eta^2}}{2b} + 2\gamma-\frac{1}{2}\right) \mathbb{E}\Bigg[ \sum_{s=1}^{S} V\left(x_s^0, x_s^{-1}\right) \Bigg] \notag
\\ & + \left(\frac{1}{16}+ \frac{{5\overline{{L}}_2^2\eta^2}}{2b}\right) \mathbb{E}\Bigg[  V\left(x_0^0, x_0^{-1}\right) \Bigg]\notag
    \\ \leq & \left(\frac{1}{16}+ \frac{5}{2}\cdot \frac{1}{64}\cdot \frac{1}{16} + \frac{2}{16} -\frac{1}{2} \right)\mathbb{E}\Bigg[ \sum_{s=0}^{S-1} \sum_{k=1}^{K-1}V\left(x_s^k, x_s^{k-1}\right) \Bigg]
      \notag
\\ & + \left(\frac{1}{16}+ \frac{5}{2}\cdot \frac{1}{64}\cdot \frac{1}{16} + \frac{2}{16} -\frac{1}{2} \right) \mathbb{E}\Bigg[ \sum_{s=1}^{S} V\left(x_s^0, x_s^{-1}\right) \Bigg] \notag
\\ & + \left(\frac{1}{16}+ \frac{{5\overline{{L}}_2^2\eta^2}}{2b}\right) \mathbb{E}\Bigg[  V\left(x_0^0, x_0^{-1}\right) \Bigg] \leq 0 
,\label{th1:res6:est1}
\end{align}
and 
\begin{align}
    \frac{5{\overline{{L}}_2^2\eta^2}}{b} & \sum_{s=0}^{S-1}\sum_{k=0}^{K-1} \Bigg[\mathbb{E}\|x_s^k- w_s\|^2\Bigg] - \frac{ \gamma}{4}\mathbb{E}\Bigg[ \sum_{s=0}^{S-1}\sum_{k=0}^{K-1}\|x_s^k-\omega_{s}\|^2 \Bigg]\notag
    \\\leq & \left( \frac{5{\overline{{L}}_2^2\eta^2}}{b} - \frac{\gamma}{4} \right) \mathbb{E}\Bigg[ \sum_{s=0}^{S-1}\sum_{k=0}^{K-1}\|x_s^{k+1}-\omega_{s}\|^2 \Bigg]\notag
    \\\leq & \left( \frac{5\gamma}{64} - \frac{\gamma}{4} \right) \mathbb{E}\Bigg[ \sum_{s=0}^{S-1}\sum_{k=0}^{K-1}\|x_s^{k+1}-\omega_{s}\|^2 \Bigg]
    \leq 0. \label{th1:res6:est2}
\end{align}
Applying (\ref{th1:res6:est1}) and (\ref{th1:res6:est2}) to (\ref{th1:res6}), we achieve
\begin{align}
\mathbb{E}\Bigg[ \max\limits_{x\in \mathcal{C}} \Bigg\{\sum_{s=0}^{S-1}\sum_{k=0}^{K-1}&\Big[\eta (g(x^{k+1}_s)-g(x)) + \eta \langle F(x^{k+1}_s), x^{k+1}_s-x \rangle \Big] \Bigg\}\Bigg]\notag
    \\
    \leq&\frac{5{\overline{{L}}_2^2\eta^2}}{b}  \sum_{s=0}^{S-1}\sum_{k=0}^{K-1} \Bigg[\mathbb{E}\|x_s^k- w_s\|^2\Bigg]+\max\limits_{x\in \mathcal{C}} V(x,x_0^0) \notag
\\
& - \frac{ \gamma}{4}\mathbb{E}\Bigg[ \sum_{s=0}^{S-1}\sum_{k=0}^{K-1}\|x_s^{k+1}-\omega_{s}\|^2 \Bigg] +(1 + \gamma K)\max\limits_{x\in \mathcal{C}} V(x,x_0^0) \notag
    \\
    \leq &\max\limits_{x\in \mathcal{C}}\left\{ V(x,x_0^0) \right\}  +(1 + \gamma K)\max\limits_{x\in \mathcal{C}} V(x,x_0^0)\notag
    \\
    = & (2+\gamma K)\max\limits_{x\in \mathcal{C}}\left\{ V(x,x_0) \right\}. \label{th1:final_est}
\end{align}

Next we use (\ref{gap}) and the fact that
\begin{align*}
\sum_{s=0}^{S-1}\sum_{k=0}^{K-1}\Big[\eta (g(x^{k+1}_s)&-g(x)) + \eta \langle F(x^{k+1}_s), x^{k+1}_s-x \rangle \Big]
    \\
    \geq& \eta \left(g\left(\sum_{s=0}^{S-1}\sum_{k=0}^{K-1}x^{k+1}_s\right)-g(x)\right) 
\\&+ \eta \left\langle F\left(\sum_{s=0}^{S-1}\sum_{k=0}^{K-1}x^{k+1}_s\right), \sum_{s=0}^{S-1}\sum_{k=0}^{K-1}x^{k+1}_s-x \right\rangle
    \\ \geq & \eta KS \left( g(x_S)-g(x) + \langle F(x_S), x_S-x \rangle \right)
\end{align*}
to get
\begin{align}
    \eta K S \mathbb{E}&\left[\operatorname{Gap} (x_S)\right]\notag
    \\ \leq & \mathbb{E}\Bigg[ \max\limits_{x\in \mathcal{C}} \Bigg\{\sum_{s=0}^{S-1}\sum_{k=0}^{K-1}\Big[\eta (g(x^{k+1}_s)-g(x)) +& \eta \langle F(x^{k+1}_s), x^{k+1}_s-x \rangle \Big] \Bigg\}\Bigg]
     .\label{th1:gap_ineq}
\end{align}
Using (\ref{th1:gap_ineq}) with (\ref{th1:final_est}), we achieve the result
\begin{align}
    \eta K S \mathbb{E}\left[\operatorname{Gap} (x_S)\right]
    \leq  (2+K\gamma)\max\limits_{x\in \mathcal{C}} V(x,x_0^0)\notag.
\end{align}
Then,
\begin{align}
    \mathbb{E}\left[\operatorname{Gap} (x_S)\right]
    \leq &  \frac{(2+K\gamma)}{\eta K S }\max\limits_{x\in \mathcal{C}} V(x,x_0^0)\label{converg}.
\end{align}
This finishes the proof.

\end{proof}

Now we are ready to obtain complexity of Algorithm \ref{alg:new}. We choice $\gamma = \frac{1}{K}$ to achieve $O(1)$ in the numerator of  estimation (\ref{converg}). In addition, it is natural to take $K = \frac{M}{3b}$ due to triple operator call in line \ref{Alg1Line7}.
\begin{corollary}
    \label{cor1}
    Let $K = \frac{M}{3b}$, $\gamma = \frac{1}{K} = \frac{3b}{M}$ and $\eta = \min \left\{\frac{\sqrt{\gamma b}}{2\bar L_2}, \frac{1}{8L_2}\right\} =\min \left\{ \frac{\sqrt{b}}{2\bar L_2}\cdot\sqrt{\frac{3b}{M}}, \frac{1}{8L_2}\right\}$ and $b \leq \frac{\sqrt{M}\bar L_2}{L_2}$. Then the total complexity of Algorithm \ref{alg:new} to reach $\epsilon$-accuracy is $\cO (M + \frac{L\sqrt{M}}{\epsilon})$.
\end{corollary}
\begin{proof}
From Theorem \ref{theorem1} it follows
\begin{align}
    \mathbb{E}\left[\operatorname{Gap} (x_S)\right]
    \leq &  \frac{(2+K\gamma)}{\eta K S }\max\limits_{x\in \mathcal{C}}\left\{ V(x,x_0) \right\}\notag
    \\ \leq & \left(\frac{2+K\cdot \frac{1}{K}}{\frac{b}{2\bar L_2}\cdot\sqrt{\frac{3}{M}}\cdot \frac{M}{3b}\cdot S} + \frac{2+K\cdot \frac{1}{K}}{\frac{1}{8 L_2}\cdot \frac{M}{3b}\cdot S}\right)\max\limits_{x\in \mathcal{C}}\left\{ V(x,x_0) \right\}\notag
    \\\leq & \left(\frac{\bar L_2\cdot 6\sqrt{3}}{\sqrt{M}S} + \frac{L_2\cdot 72b}{NS}\right)\max\limits_{x\in \mathcal{C}}\left\{ V(x,x_0) \right\}\notag
    \\= & \cO\left( \frac{\bar L_2}{\sqrt{M}S} +  \frac{L_2\cdot b}{MS}\right)
    \notag.
    \end{align}
With $b \leq \frac{\sqrt{M}\bar L_2}{L_2}$ we have $\mathbb{E}\left[\operatorname{Gap} (x_S)\right] = \cO\left(\frac{\bar L_2}{\sqrt{M}S}\right)$.
One outer iteration requires $\cO(M)$ evaluations of $F_m$. Hence, the final complexity of $\cO (M + \frac{\bar L_2\sqrt{M}}{\epsilon})$.
\end{proof}

Now let us consider the case when Lipschitzness of the operator is given in the form of Assumption \ref{ass3b}.

\begin{theorem}\label{theorem2}
Consider the problem \eqref{eq:VI}+\eqref{eq:sum} under Assumptions~\ref{ass1}, \ref{ass2} and \ref{ass3b}. Let  $\{x_s^k\}$ be the sequence generated by Algorithm~\ref{alg:new} with tuning of $\eta, \theta, \alpha, \beta, \gamma$  as follows:
\begin{align}
    0 < \gamma = p \leq \frac{1}{16}, \notag 
\end{align} 
\begin{align}
    \quad \eta = \min\left\{\frac{\sqrt{\gamma b}}{8 L \sqrt{1 + C \ln n}}, \frac{1}{8L \sqrt{1 + C \ln n}}\right\}. \label{th2:eta}
\end{align} 
Then for $x_S = \frac{1}{KS} \sum\limits_{s=0}^{S-1}\sum\limits_{k=0}^{K-1} x_s^k$ it holds that
\begin{align}
    \mathbb{E}\left[\operatorname{Gap} (x_S)\right]
    \leq &  \frac{(2+K\gamma)}{\eta K S }\max\limits_{x\in \mathcal{C}}\left\{ V(x,x_0) \right\}\notag.
\end{align}
\end{theorem}
\begin{proof}
The beginning of the proof is similar to the proof of Theorem \ref{theorem1}.
With given assumptions we replace (\ref{th1:res1:Young}) and (\ref{th1:res2:Young}) with

\begin{align}
    -\eta \langle F(x_s^k)- F(x_s^{k-1}), x_s^{k+1}-x_s^k \rangle \leq & 2\eta^2 \|F(x_s^k)- F(x_s^{k-1})\|_*^2 + \frac{1}{8} \| x_s^{k+1}-x_s^k \|^2 \notag 
    \\\leq & 2L^2\eta^2 \|x_s^k- x_s^{k-1}\|^2 + \frac{1}{8} \| x_s^{k+1}-x_s^k \|^2 \notag
    \\ \leq & \frac{1}{32} \|x_s^k- x_s^{k-1}\|^2 + \frac{1}{8} \| x_s^{k+1}-x_s^k \|^2 \notag
\end{align}
and
\begin{align}
     \eta \langle F(x^{0}_S)-F(x^{-1}_S), x^0_S-x\rangle \leq & \frac{\eta^2}{2}\|F(x_S^0)- F(x_S^{-1})\|_*^2 + \frac{1}{2}\|x_S^0 - x\|^2 \notag
     \\\leq & \frac{\eta^2 L^2}{2}\|x_S^{0}-x_S^{-1}\|^2+\frac{1}{2}\|x_S^0 - x\|^2\notag
     \\ \leq &\frac{1}{128}\|x_S^{0}-x_S^{-1}\|^2+\frac{1}{2}\|x_S^0- x\|^2.\notag
\end{align}
These inequalities are correct due to Young's inequality (\ref{ap:Young}), Assumption \ref{ass3b} and the choice of step (\ref{th2:eta}).

In a similar way to (\ref{th1:res4}) from Theorem \ref{theorem1} we get the result 

\begin{align}
\mathbb{E}\Bigg[ \max\limits_{x\in \mathcal{C}} \Bigg\{\sum_{s=0}^{S-1}\sum_{k=0}^{K-1}&\Big[\eta (g(x^{k+1}_s)-g(x)) + \eta \langle F(x^{k+1}_s), x^{k+1}_s-x \rangle \Big] \Bigg\}\Bigg]\notag
    \\ 
    \leq &
    \frac{1}{128}\mathbb{E}\Big[\|x_{S-1}^{K}-x_{S-1}^{K-1}\|^2\Big] + \frac{1}{32} \mathbb{E}\Bigg[ \sum_{s=0}^{S-1} \sum_{k=0}^{K-1}\|x_s^k- x_s^{k-1}\|^2 \Bigg]\notag
\\
&+ \left(\gamma-\frac{3}{4}\right) \mathbb{E}\Bigg[ \sum_{s=0}^{S-1}\sum_{k=0}^{K-1} V(x_s^{k+1},x_s^k) \Bigg]\notag
\\
& +  \eta \mathbb{E}\Bigg[ \sum_{s=0}^{S-1}\sum_{k=0}^{K-1}\Big[ \langle\mathbb{E}_k\left[\Delta^k\right]-\Delta^k, x^{k+1}_s-x^k_s\rangle\big] \Bigg]\notag
\\
&+ \eta \mathbb{E}\Bigg[ \max\limits_{x\in \mathcal{C}} \left\{ \sum_{s=0}^{S-1} \sum_{k=0}^{K-1}\Big[\langle\mathbb{E}_k\left[\Delta^k\right]-\Delta^k, x\rangle\Big] \right\} \Bigg] \notag
\\& - \frac{ \gamma}{2}\mathbb{E}\Bigg[ \sum_{s=0}^{S-1}\sum_{k=0}^{K-1}\|x_s^{k+1}-\omega_{s}\|^2 \Bigg]+(1 + \gamma K)\max\limits_{x\in \mathcal{C}} V(x,x_0^0). \label{th2:res4}
\end{align}
Similar to (\ref{th1:res4:est}), we get \begin{align}
    \eta \mathbb{E}\Bigg[ \sum_{s=0}^{S-1}\sum_{k=0}^{K-1}&\Big[ \langle\mathbb{E}_k\left[\Delta^k\right]-\Delta^k, x^{k+1}_s-x^k_s\rangle\big] \Bigg]\notag
    \\
    \leq&2\eta^2  \sum_{s=0}^{S-1}\sum_{k=0}^{K-1} \mathbb{E} \|\mathbb{E}_k\left[\Delta^k\right]-\Delta^k \|^2_*
+\frac{1}{8}\sum_{s=0}^{S-1}\sum_{k=0}^{K-1} \Bigg[ \mathbb{E}\left[\|x_s^{k+1} - x_s^k\|^2\right] \Bigg].
\label{th2:res4:est}
\end{align}

$\mathbb{E}_k\left[\Delta^k\right]-\Delta^k$ is a stochastic process with $\mathbb{E}[\mathbb{E}_k\left[\Delta^k\right]-\Delta^k]=0$. Therefore, according to Lemma \ref{lem:3b}
\begin{align}
    \eta \mathbb{E}\Bigg[\max\limits_{x\in \mathcal{C}} \Bigg\{ \sum_{s=0}^{S-1}\sum_{k=0}^{K-1}& \langle\mathbb{E}_k\left[\Delta^k\right]-\Delta^k, x\rangle \Bigg\} \Bigg] \notag
    \\\leq &
    V(x, x^0_0)
+ \frac{\eta^2}{2}  \sum_{s=0}^{S-1}\sum_{k=0}^{K-1} \mathbb{E} \|\mathbb{E}_k\left[\Delta^k\right]-\Delta^k \|^2_*.\label{th2:res4:lem3}
\end{align}
Applying (\ref{th2:res4:est}) and (\ref{th2:res4:lem3}) to (\ref{th2:res4}), we get
\begin{align}
\mathbb{E}\Bigg[ \max\limits_{x\in \mathcal{C}} \Bigg\{\sum_{s=0}^{S-1}&\sum_{k=0}^{K-1}\Big[\eta (g(x^{k+1}_s)-g(x)) + \eta \langle F(x^{k+1}_s), x^{k+1}_s-x \rangle \Big] \Bigg\}\Bigg]\notag
    \\
    \leq&
    \frac{1}{128}\mathbb{E}\Big[\|x_{S-1}^{K}-x_{S-1}^{K-1}\|^2\Big]+\frac{1}{32} \mathbb{E}\Bigg[ \sum_{s=0}^{S-1} \sum_{k=0}^{K-1}\|x_s^k- x_s^{k-1}\|^2 \Bigg]\notag
\\
&+ \left(\gamma-\frac{3}{4}\right) \mathbb{E}\Bigg[ \sum_{s=0}^{S-1}\sum_{k=0}^{K-1} V(x_s^{k+1},x_s^k) \Bigg]\notag
\\
&+2\eta^2  \sum_{s=0}^{S-1}\sum_{k=0}^{K-1} \mathbb{E} \|\mathbb{E}_k\left[\Delta^k\right]-\Delta^k \|^2_*
+\frac{1}{8}\sum_{s=0}^{S-1}\sum_{k=0}^{K-1} \Bigg[ \mathbb{E}\left[\|x_s^{k+1} - x_s^k\|^2\right] \Bigg]\notag
\\
&+V(x, x^0_0)
+ \frac{\eta^2}{2}  \sum_{s=0}^{S-1}\sum_{k=0}^{K-1} \mathbb{E} \|\mathbb{E}_k\left[\Delta^k\right]-\Delta^k \|^2_*\notag
\\
& - \frac{ \gamma}{2}\mathbb{E}\Bigg[ \sum_{s=0}^{S-1}\sum_{k=0}^{K-1}\|x_s^{k+1}-\omega_{s}\|^2 \Bigg]+(1 + \gamma K)\max\limits_{x\in \mathcal{C}} V(x,x_0^0) \notag
    \\
    \leq&
    \frac{1}{128}\mathbb{E}\Big[\|x_{S-1}^{K}-x_{S-1}^{K-1}\|^2\Big] +\frac{1}{32} \mathbb{E}\Bigg[ \sum_{s=0}^{S-1} \sum_{k=0}^{K-1}\|x_s^k- x_s^{k-1}\|^2 \Bigg]\notag
\\
&+ \left(\gamma-\frac{3}{4}\right) \mathbb{E}\Bigg[ \sum_{s=0}^{S-1}\sum_{k=0}^{K-1} V(x_s^{k+1},x_s^k) \Bigg]\notag
\\&+\frac{5\eta^2}{2}  \sum_{s=0}^{S-1}\sum_{k=0}^{K-1} \mathbb{E} \|\mathbb{E}_k\left[\Delta^k\right]-\Delta^k \|^2_*\notag
\\&
+\frac{1}{8}\sum_{s=0}^{S-1}\sum_{k=0}^{K-1} \Bigg[ \mathbb{E}\left[\|x_s^{k+1} - x_s^k\|^2\right] \Bigg]+ \frac{1}{2} \max\limits_{x\in \mathcal{C}} \|x - x^0_0\|^2 \notag
\\&
- \frac{ \gamma}{2}\mathbb{E}\Bigg[ \sum_{s=0}^{S-1}\sum_{k=0}^{K-1}\|x_s^{k+1}-\omega_{s}\|^2 \Bigg]+(1 + \gamma K)\max\limits_{x\in \mathcal{C}} V(x,x_0^0). \label{th2:before_CSineq}
\end{align}

Applying Lemma \ref{lem:2b}, the fact (\ref{main_inequality}), property $\|\cdot \|_2\leq \|\cdot \|$ and (\ref{th1:CSineq}) to (\ref{th2:before_CSineq}), we achieve
\begin{align}
\mathbb{E}\Bigg[ \max\limits_{x\in \mathcal{C}}& \Bigg\{\sum_{s=0}^{S-1}\sum_{k=0}^{K-1}\Big[\eta (g(x^{k+1}_s)-g(x)) + \eta \langle F(x^{k+1}_s), x^{k+1}_s-x \rangle \Big] \Bigg\}\Bigg]\notag
    \\
    \leq &
    \frac{1}{128}\mathbb{E}\Bigg[\|x_{S-1}^{K}-x_{S-1}^{K-1}\|^2\Big] + \frac{1}{32} \mathbb{E}\Bigg[ \sum_{s=0}^{S-1} \sum_{k=0}^{K-1}\|x_s^k- x_s^{k-1}\|^2 \Bigg]\notag
\\
&+ \left(\gamma-\frac{3}{4}\right) \mathbb{E}\Bigg[ \sum_{s=0}^{S-1}\sum_{k=0}^{K-1} V(x_s^{k+1},x_s^k) \Bigg]\notag
\\
&+\frac{5\eta^2}{2}   \sum_{s=0}^{S-1}\sum_{k=0}^{K-1} \Bigg[{\frac{2(1+C\ln n){L}^{2}}{b}}\mathbb{E}\left[\|x^{k}_s-w_{s}\|^2+\|x^{k}_s - x^{k-1}_s\|^2 \right]\Bigg]\notag
\\
&+\frac{1}{8}\sum_{s=0}^{S-1}\sum_{k=0}^{K-1} \Bigg[ \mathbb{E}\left[\|x_s^{k+1} - x_s^k\|^2_2\right] \Bigg]+ \frac{1}{2} \max\limits_{x\in \mathcal{C}} \|x - x^0_0\|^2_2  \notag
\\
& - \frac{ \gamma}{2}\mathbb{E}\Bigg[ \sum_{s=0}^{S-1}\sum_{k=0}^{K-1}\|x_s^{k+1}-\omega_{s}\|^2 \Bigg]+(1 + \gamma K)\max\limits_{x\in \mathcal{C}} V(x,x_0^0)\notag
    \\
    \leq &
    \frac{1}{64}\mathbb{E} \left[ V(x_{S-1}^{K},x_{S-1}^{K-1})\right] + \frac{1}{16} \mathbb{E}\Bigg[ \sum_{s=0}^{S-1} \sum_{k=0}^{K-1}V(x_s^k, x_s^{k-1}) \Bigg]\notag
\\
&+ \left(\gamma-\frac{3}{4}\right) \mathbb{E}\Bigg[ \sum_{s=0}^{S-1}\sum_{k=0}^{K-1} V(x_s^{k+1},x_s^k) \Bigg] +\frac{1}{4}\sum_{s=0}^{S-1}\sum_{k=0}^{K-1} \Bigg[ \mathbb{E}\left[V(x_s^{k+1}, x_s^k)\right] \Bigg]\notag
\\
&+\frac{5\eta^2}{2}   \sum_{s=0}^{S-1}\sum_{k=0}^{K-1} \Bigg[{\frac{{L}^{2}}{b}}\mathbb{E}\Bigg[2(1+C\ln n)\|x^{k}_s-w_{s}\|^2\notag
\\&+\frac{2(1+C\ln n)}{2}V\left(x^{k}_s,x^{k-1}_s\right)\Bigg]\Bigg]+ \max\limits_{x\in \mathcal{C}} V(x,x_0^0)\notag
\\
&- \frac{ \gamma}{4}\mathbb{E}\Bigg[ \sum_{s=0}^{S-1}\sum_{k=0}^{K-1}\|x_s^k-\omega_{s}\|^2 \Bigg]+ \frac{ \gamma}{2} \mathbb{E}\Bigg[ \sum_{s=0}^{S-1}\sum_{k=0}^{K-1}\|x_s^{k+1}-x_s^k\|^2 \Bigg]\notag\notag
\\&+(1 + \gamma K)\max\limits_{x\in \mathcal{C}} V(x,x_0^0)\notag
    \\
    \leq&
    \frac{1}{64}\mathbb{E} \left[ V(x_{S-1}^{K},x_{S-1}^{K-1})\right]\notag
\\& +\left(\frac{1}{16}    + \frac{{10(1+C\ln n){L}^2\eta^2}}{4b}\right) \mathbb{E}\Bigg[ \sum_{s=0}^{S-1} \sum_{k=0}^{K-1}V\left(x_s^k, x_s^{k-1}\right)^2 \Bigg]\notag
\\
&+ \left(\gamma-\frac{1}{2}\right) \mathbb{E}\Bigg[ \sum_{s=0}^{S-1}\sum_{k=0}^{K-1} V(x_s^{k+1},x_s^k) \Bigg]\notag
\\
&+\frac{10(1+C\ln n){L^2\eta^2}}{2b}  \sum_{s=0}^{S-1}\sum_{k=0}^{K-1} \Bigg[\mathbb{E}\|x_s^k- w_s\|^2\Bigg]+\max\limits_{x\in \mathcal{C}} V(x,x_0^0)\notag
\\
&- \frac{ \gamma}{4}\mathbb{E}\Bigg[ \sum_{s=0}^{S-1}\sum_{k=0}^{K-1}\|x_s^k-\omega_{s}\|^2 \Bigg] + \gamma \mathbb{E}\Bigg[ \sum_{s=0}^{S-1}\sum_{k=0}^{K-1}V(x_s^{k+1},x_s^k) \Bigg]\notag
\\
&+(1 + \gamma K)\max\limits_{x\in \mathcal{C}} V(x,x_0^0). \notag
\end{align}
The following steps are similar up to (\ref{th1:res6}).
We use the fact (\ref{main_inequality}), 
inequality (\ref{th1:gamma}), lines \ref{Alg1Line12} and \ref{Alg1Line14} of Algorithm \ref{alg:new}, the choice of the step (\ref{th1:eta}) and similar calculations to (\ref{th1:res6:est1}) and (\ref{th1:res6:est2})  to get
\begin{align}
\frac{1}{64}\mathbb{E} \Big[ V( & x_{S-1}^{K},x_{S-1}^{K-1})\Big]+\left(\frac{1}{16}+ \frac{{10(1+C\ln n)L^2\eta^2}}{4b}\right) \mathbb{E}\Bigg[ \sum_{s=0}^{S-1} \sum_{k=0}^{K-1}V\left(x_s^k, x_s^{k-1}\right) \Bigg]\notag
\\&+ \left(\gamma-\frac{1}{2}\right) \mathbb{E}\Bigg[ \sum_{s=0}^{S-1}\sum_{k=0}^{K-1} V(x_s^{k+1},x_s^k) \Bigg]+\gamma \mathbb{E}\Bigg[ \sum_{s=0}^{S-1}\sum_{k=0}^{K-1}V(x_s^{k+1},x_s^k) \Bigg]\notag
    \\ \leq & \left(\frac{1}{16}+ \frac{{10(1+C\ln n)L^2\eta^2}}{4b} + 2\gamma-\frac{1}{2} \right)\mathbb{E}\Bigg[ \sum_{s=0}^{S-1} \sum_{k=1}^{K-1}V\left(x_s^k, x_s^{k-1}\right) \Bigg]
      \notag
\\ & + \left(\frac{1}{16}+ \frac{{10(1+C\ln n)L^2\eta^2}}{4b} + 2\gamma-\frac{1}{2}\right) \mathbb{E}\Bigg[ \sum_{s=1}^{S} V\left(x_s^0, x_s^{-1}\right) \Bigg] \notag
\\ & + \left(\frac{1}{16}+ \frac{{10(1+C\ln n)L^2\eta^2}}{4b}\right) \mathbb{E}\Bigg[  V\left(x_0^0, x_0^{-1}\right) \Bigg]\notag
    \\ \leq & \left(\frac{1}{16}+ \frac{10}{4}\cdot \frac{1}{64}\cdot \frac{1}{16} + \frac{2}{16} -\frac{1}{2} \right)\mathbb{E}\Bigg[ \sum_{s=0}^{S-1} \sum_{k=1}^{K-1}V\left(x_s^k, x_s^{k-1}\right) \Bigg]
      \notag
\\ & + \left(\frac{1}{16}+ \frac{10}{4}\cdot \frac{1}{64}\cdot \frac{1}{16} + \frac{2}{16} -\frac{1}{2} \right) \mathbb{E}\Bigg[ \sum_{s=1}^{S} V\left(x_s^0, x_s^{-1}\right) \Bigg] \notag
\\ & + \left(\frac{1}{16}+ \frac{{5(1+C\ln n)L^2\eta^2}}{2b}\right) \mathbb{E}\Bigg[  V\left(x_0^0, x_0^{-1}\right) \Bigg] \leq 0 
,\notag 
\end{align}
and 
\begin{align}
    \frac{{10(1+C\ln n)L^2\eta^2}}{2b} & \sum_{s=0}^{S-1}\sum_{k=0}^{K-1} \Bigg[\mathbb{E}\|x_s^k- w_s\|^2\Bigg] - \frac{ \gamma}{4}\mathbb{E}\Bigg[ \sum_{s=0}^{S-1}\sum_{k=0}^{K-1}\|x_s^k-\omega_{s}\|^2 \Bigg]\notag
    \\\leq & \left( \frac{{10(1+C\ln n)L^2\eta^2}}{2b} - \frac{\gamma}{4} \right) \mathbb{E}\Bigg[ \sum_{s=0}^{S-1}\sum_{k=0}^{K-1}\|x_s^{k+1}-\omega_{s}\|^2 \Bigg]\notag
    \\\leq & \left( \frac{5\gamma}{64} - \frac{\gamma}{4} \right) \mathbb{E}\Bigg[ \sum_{s=0}^{S-1}\sum_{k=0}^{K-1}\|x_s^{k+1}-\omega_{s}\|^2 \Bigg]
    \leq 0. \notag 
\end{align}
Steps similar to the proof of Theorem \ref{theorem1} finish the proof.
\end{proof}

Now we proceed similarly to Corollary \ref{cor1}.
\begin{corollary}\label{cor2}
    Let $K = \frac{M}{3b}$, $\gamma = \frac{1}{K} = \frac{3b}{M}$ and $\eta = \min \left\{\frac{\sqrt{\gamma b}}{2\sqrt{1 + C \ln n} L}, \frac{1}{8L\sqrt{1 + C \ln n}}\right\} =\min \left\{ \frac{b}{2\sqrt{1 + C \ln n}L}\cdot\sqrt{\frac{3}{M}}, \frac{1}{8L\sqrt{1 + C \ln n}}\right\}$ and $b \leq \sqrt{M}$. Then the total complexity of the Algorithm \ref{alg:new} to reach $\epsilon$-accuracy is $\cO (M + \frac{L\sqrt{M}}{\epsilon})$.
\end{corollary}
\begin{proof}
From Theorem \ref{theorem2} it follows
\begin{align}
    \mathbb{E}\left[\operatorname{Gap} (x_S)\right]
    \leq &  \frac{(2+K\gamma)}{\eta K S }\max\limits_{x\in \mathcal{C}}\left\{ V(x,x_0) \right\}\notag
    \\ \leq & \left(\frac{2+K\cdot \frac{1}{K}}{\frac{b}{2 L}\cdot\sqrt{\frac{3}{M}}\cdot \frac{M}{3b}\cdot S} + \frac{2+K\cdot \frac{1}{K}}{\frac{1}{8 L}\cdot\frac{M}{3b}\cdot S}\right)\sqrt{1 + C \ln n}\max\limits_{x\in \mathcal{C}}\left\{ V(x,x_0) \right\}\notag
    \\\leq & \left(\frac{ L\cdot 6\sqrt{3}}{\sqrt{M}S} + \frac{ L\cdot 72b}{MS}\right)\sqrt{1 + C \ln n}\cdot \max\limits_{x\in \mathcal{C}}\left\{ V(x,x_0) \right\}\notag
    \\= & \mathcal{\tilde O}\left( \frac{ L}{\sqrt{M}S} + \frac{L\cdot b}{MS} \right)
    \notag.
    \end{align}
With $b \leq \sqrt{M}$ we have $\mathbb{E}\left[\operatorname{Gap} (x_S)\right] = \mathcal{\tilde O}\left(\frac{L}{\sqrt{M}S}\right)$.
On epoch requires $\cO(N)$ evaluations $F_m$. Hence, the final complexity of $\mathcal{\tilde O} (M + \frac{ L\sqrt{M}}{\epsilon})$.
\end{proof}

As mentioned before, $L\leq \bar L_2$. Thus, we can state $\frac{L}{\sqrt{M}S}\leq  \frac{\bar L_2}{\sqrt{M}S}$.

\section{Experiments}
In this section, empirical performance of Algorithm \ref{alg:new} is shown. Similar to the previous works \cite{alacaoglu2021stochastic,pichugin2024optimal}, we consider matrix games  
 \begin{align}
    \min\limits_{x\in \mathcal{X}} \max\limits_{y\in \mathcal{Y}}  \left\langle Ax, y \right\rangle
    \label{bilinear} 
\end{align}
with matrix $A\in \mathbb{R}^{n\times n}$ and use simplex constraints (\ref{simplex}) in the entropic setup:
\begin{align*}
    \mathcal{X} = \Delta^n,  \quad  \mathcal{Y} = \Delta^n.
\end{align*}
For our experiments we choose the policeman and burglar matrix \cite{nemirovski2013mini} and the first test matrix from \cite{nemirovski2009robust}.

Note that the problem \eqref{bilinear} does not have the finite-sum form as \eqref{eq:VI}+\eqref{eq:sum}, but it can be rewritten as 
\begin{align}
A = \sum_{i=1}^n A_{i:} \text{ or } A = \sum_{j=1}^n A_{:j}. 
\end{align}
where $A_{i:}$ is the $i$-th row of $A$ and $A_{:j}$ is the $j$-th column of $A$ -- see details in Section 5.1.2 from \cite{alacaoglu2021stochastic}. In this formulation the problem already has the form of a finite sum.

In order to 
calculate the update (in particular, we need to compute $\nabla h$ in the closed form) we introduce  $u = (u^x, u^y)$ and $v = (v^x, v^y)$.
With this notation we can adapt (\ref{entropy}) in a similar to the previous works \cite{alacaoglu2021forward,pichugin2024optimal,carmon2019variance} way:
\begin{align*}
    \nabla h \left(x^{k+1}_s\right) =& (1 - \gamma)\nabla h\left(x^k_s\right) + \gamma \nabla h \left(\overline{u}^s\right) 
\\&
- \eta\left( \frac{1}{b}\sum\limits_{i\in B^k}\left[ \frac{1}{r_i}A_{i:}\left(2y^k_{s, i} - v_{s, i} - y^{k-1}_{s, i}\right)\right] + A^\top v_s  \right)
    \\=& (1 - \gamma)\nabla h\left(x^k_s \right) + \gamma \nabla h \left(\overline{u}^s\right)
-\eta A^\top v_s
\\&- \frac{\eta}{b}\sum\limits_{i\in B^k}\left[ \frac{1}{r_i}A_{i:}\left\| 2y^k_{s} - v_{s} - y^{k-1}_{s} \right\| \cdot \operatorname{sign} \left(2y^k_{s, i} - v_{s, i} - y^{k-1}_{s, i}\right)\right].
\end{align*}

Likewise, for the second component one can write
\begin{align*}
    \nabla h \left(y^{k+1}_s\right)=& (1 - \gamma)\nabla h \left(y^k_s\right) + \gamma \nabla h \left(\overline{v}^s\right)
    +\eta A u_s
\\&+ \frac{\eta}{b}\sum\limits_{j\in B^k}\left[ \frac{1}{c_j}A_{:j}\left\| 2x^k_{s} - u_{s} - x^{k-1}_{s} \right\|\cdot \operatorname{sign} \left(2x^k_{s, j} - u_{s, j} - x^{k-1}_{s, j}\right)\right].
\end{align*}

Now we are ready to implement the Algorithm \ref{alg:new} and compare our method with other methods. In particular, we choose Algorithm 2 from \cite{alacaoglu2021stochastic}, Algorithm 1+2 from \cite{carmon2019variance} --- state-of-the-art competitors (see Table 1). The parameters of the algorithms are taken from the provided papers.

We run all methods with various batch sizes and use theoretical parameters (see Theorems \ref{theorem1} and \ref{theorem2} for our method and Section 6 from \cite{alacaoglu2021stochastic} for competitors). Duality gap \eqref{gap} is used as the convergence measure, in the matrix games setup it can be simply computed as $\left[\max_i (A^\top x)_i - \min_j (Ay)_j\right]$ for simplex constraints. The comparison criterion is the number of oracle calls (one call is computationally equal to calculations of $A_{i:}y_i$ and $A_{j:}^\top x_j$). 
The results are reflected in Figures \ref{fig:nem} and \ref{fig:PB}. It is clearly shown that Algorithm \ref{alg:new} outperforms competitors and its convergence rate is insensitive to the size of batch (unlike the other methods).

\begin{figure}
\centering
\begin{minipage}[][][b]{\textwidth}
\centering
\includegraphics[width=0.35\textwidth]{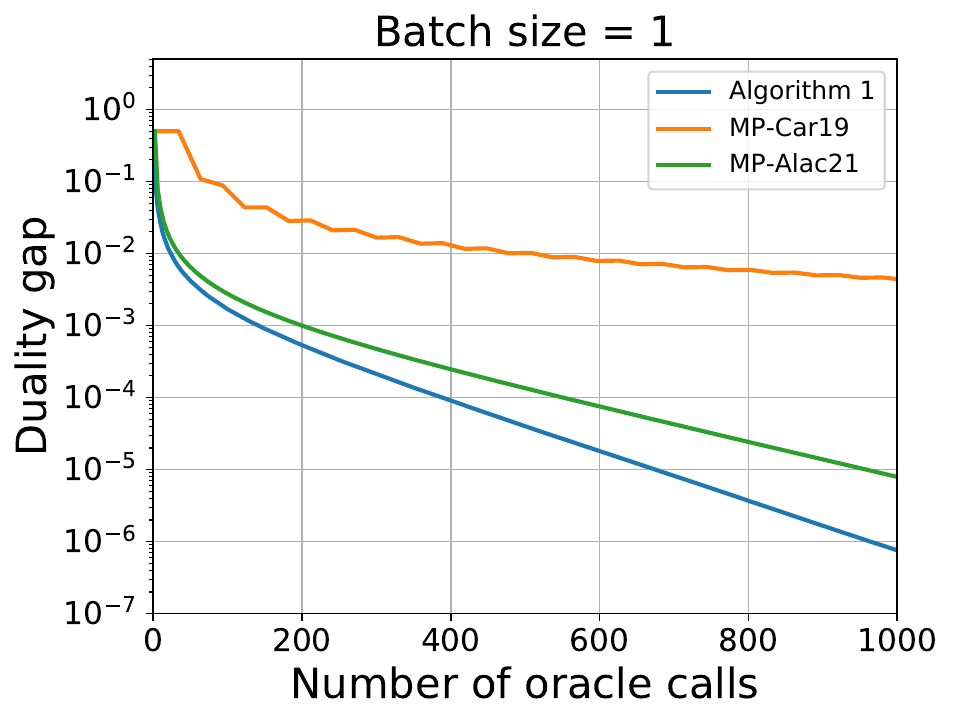}
\includegraphics[width=0.35\textwidth]{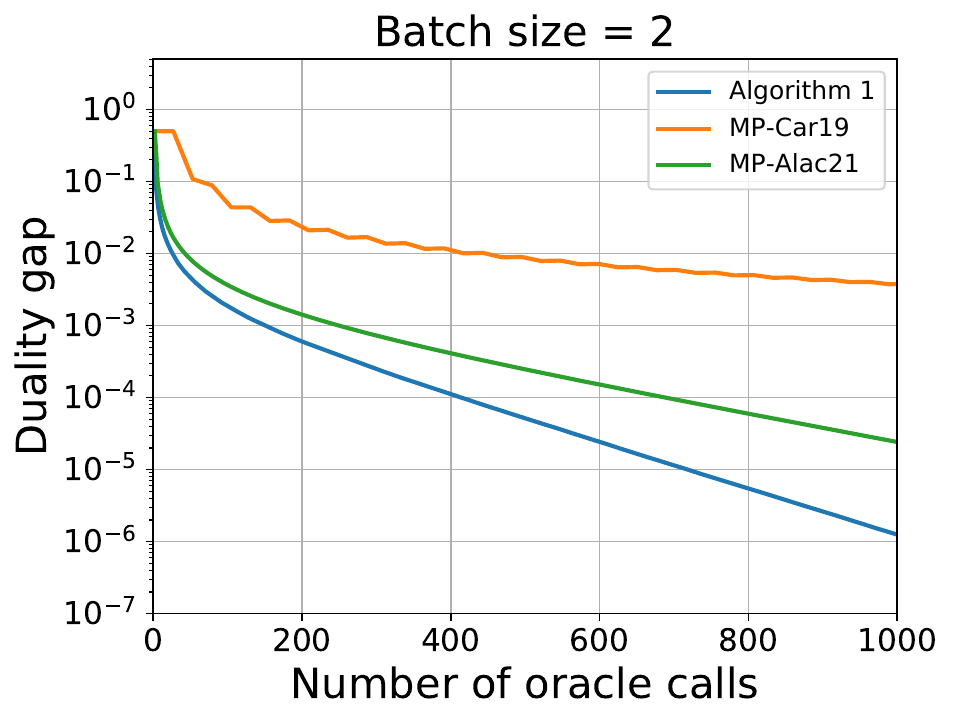}
\includegraphics[width=0.35\textwidth]{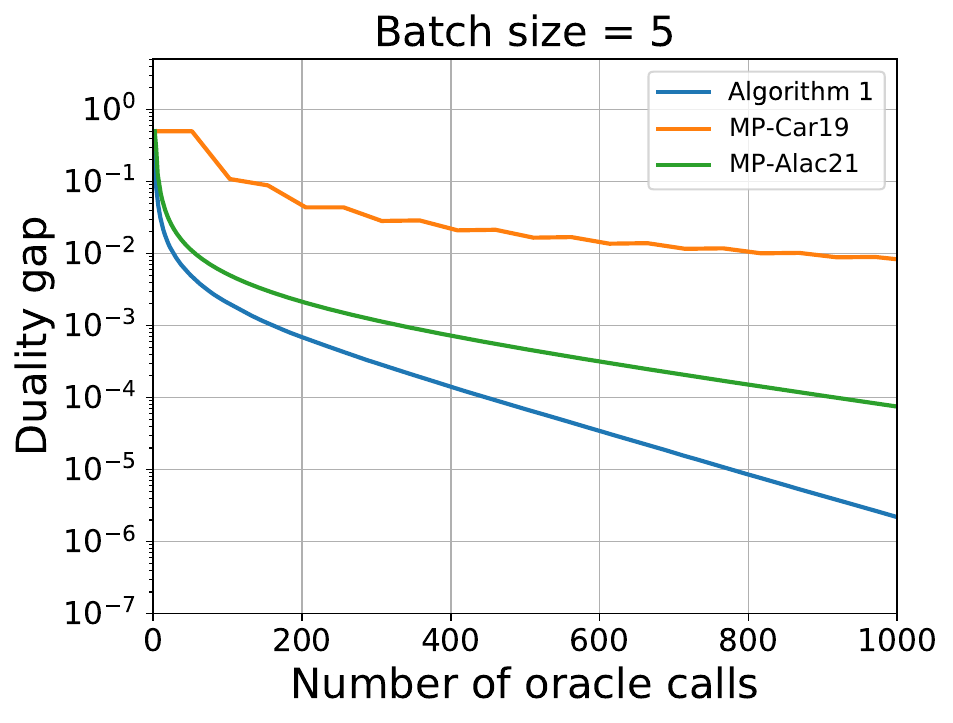}
\includegraphics[width=0.35\textwidth]{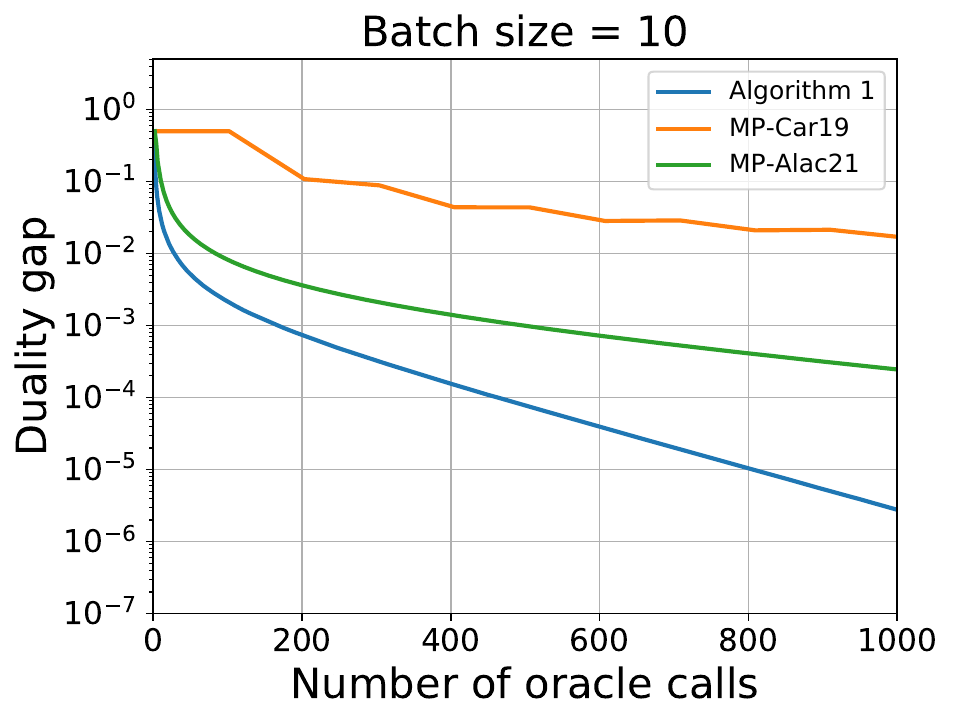}
\end{minipage}
\caption{Comparison of computational complexities for Algorithm~\ref{alg:new}, \texttt{MP-Car19} (Algorithm 1+2 from \cite{carmon2019variance}), and \texttt{MP-Alac21} (Algorithm 2 \cite{alacaoglu2021stochastic}) with different batch sizes on test matrix from \cite{nemirovski2013mini}.}
\label{fig:nem}
\end{figure}

\begin{figure}
\centering
\begin{minipage}[][][b]{\textwidth}
\centering
\includegraphics[width=0.35\textwidth]{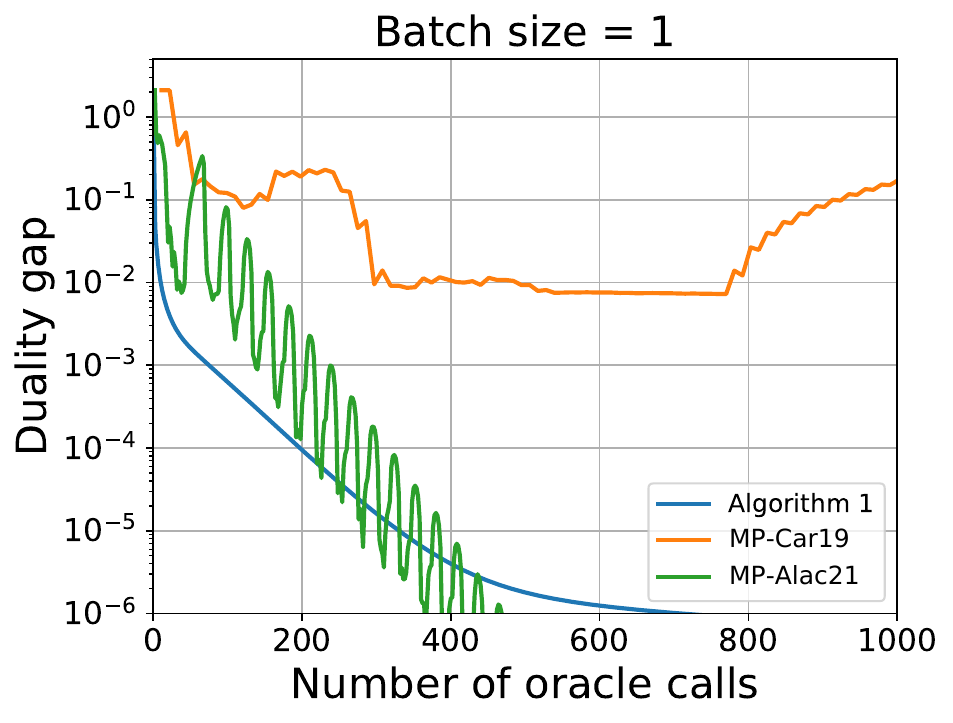}
\includegraphics[width=0.35\textwidth]{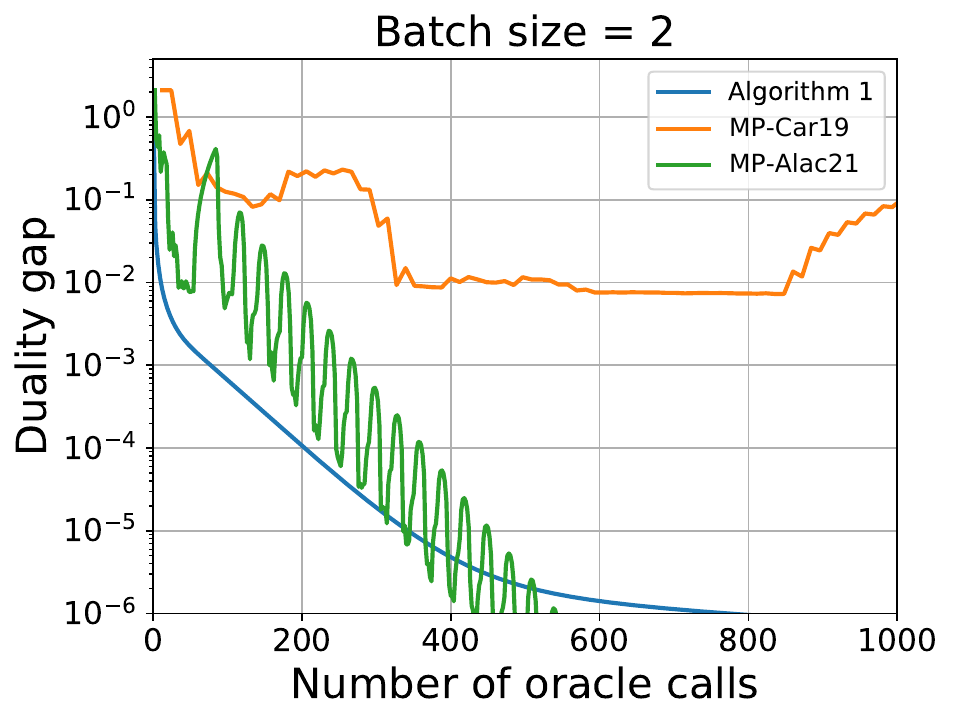}
\includegraphics[width=0.35\textwidth]{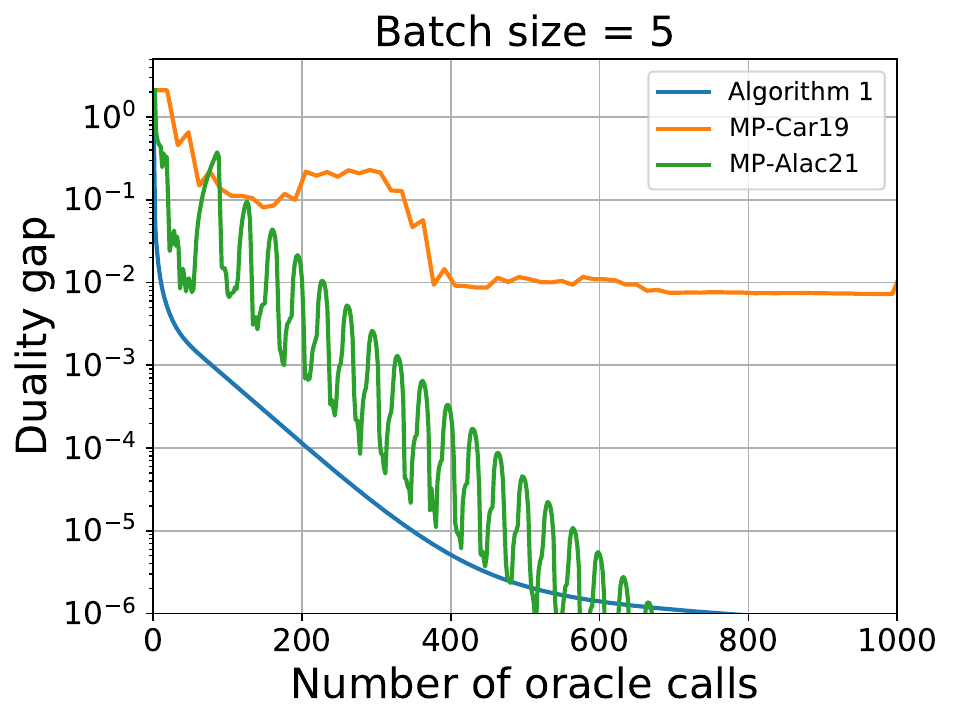}
\includegraphics[width=0.35\textwidth]{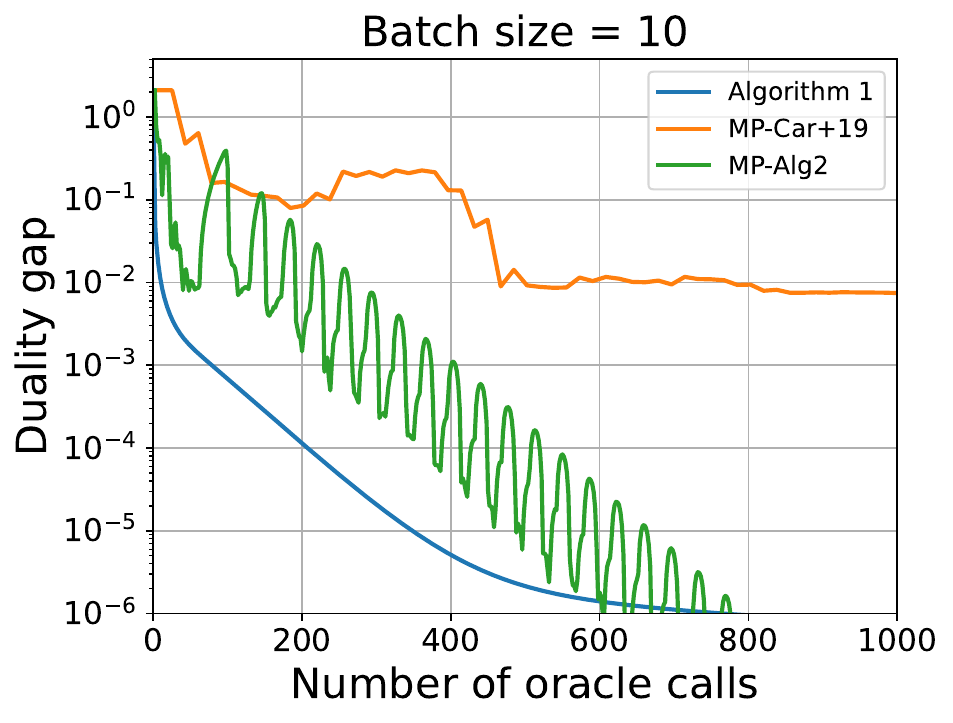}
\end{minipage}
\caption{Comparison of computational complexities for Algorithm~\ref{alg:new}, \texttt{MP-Car19} (Algorithm 1+2 from \cite{carmon2019variance}), and \texttt{MP-Alac21} (Algorithm 2 \cite{alacaoglu2021stochastic}) with different batch sizes on policeman and burglar matrix from \cite{nemirovski2013mini}.}
\label{fig:PB}
\end{figure}
\begin{figure}
\centering
\begin{minipage}[][][b]{\textwidth}
\centering
\includegraphics[width=0.35\textwidth]{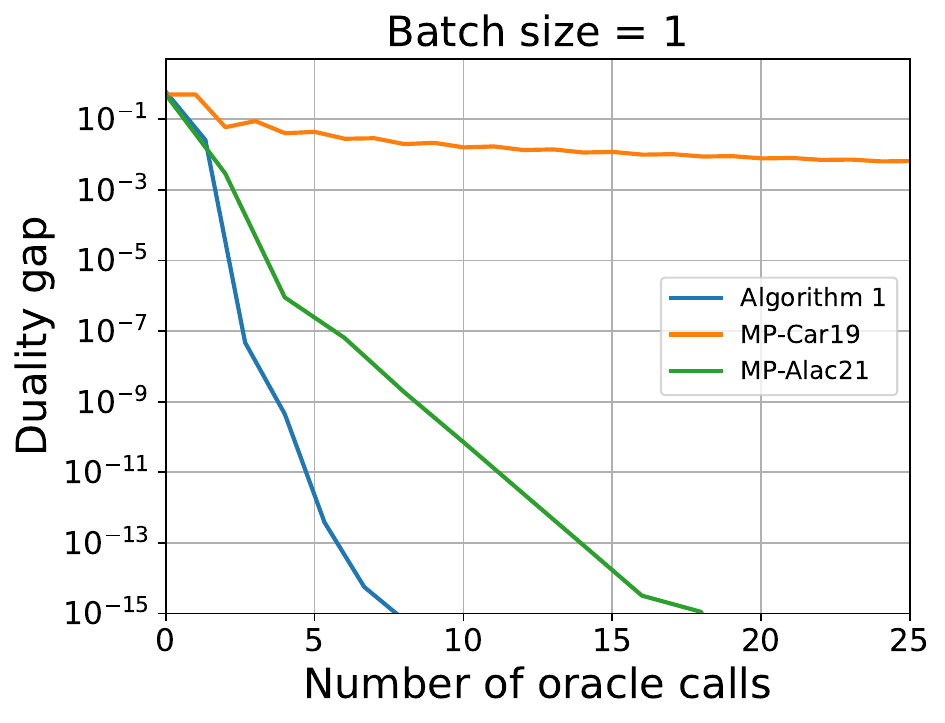}
\includegraphics[width=0.35\textwidth]{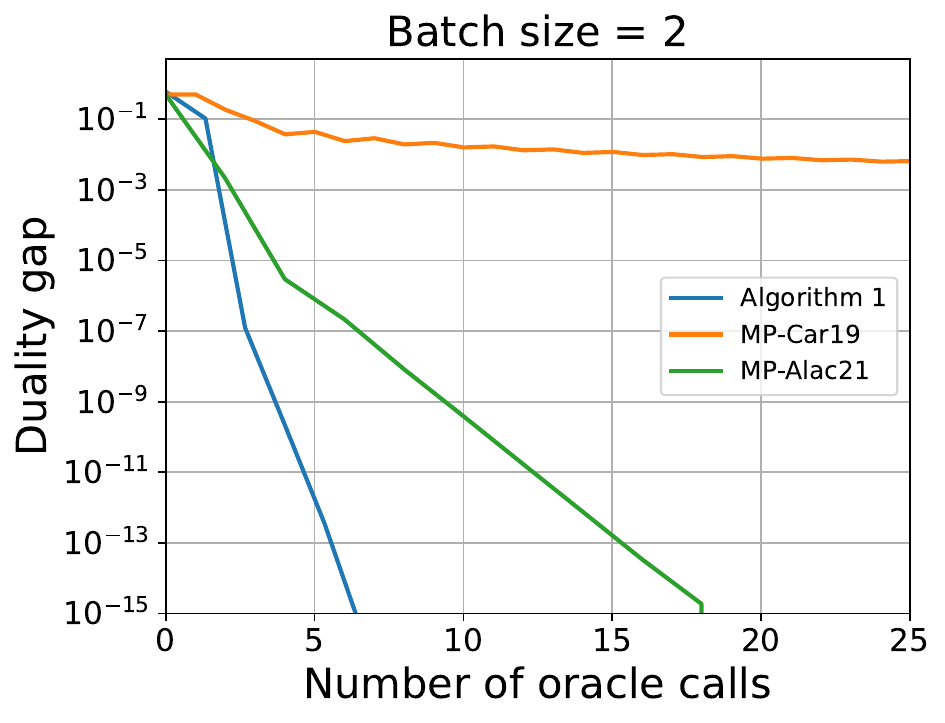}
\includegraphics[width=0.35\textwidth]{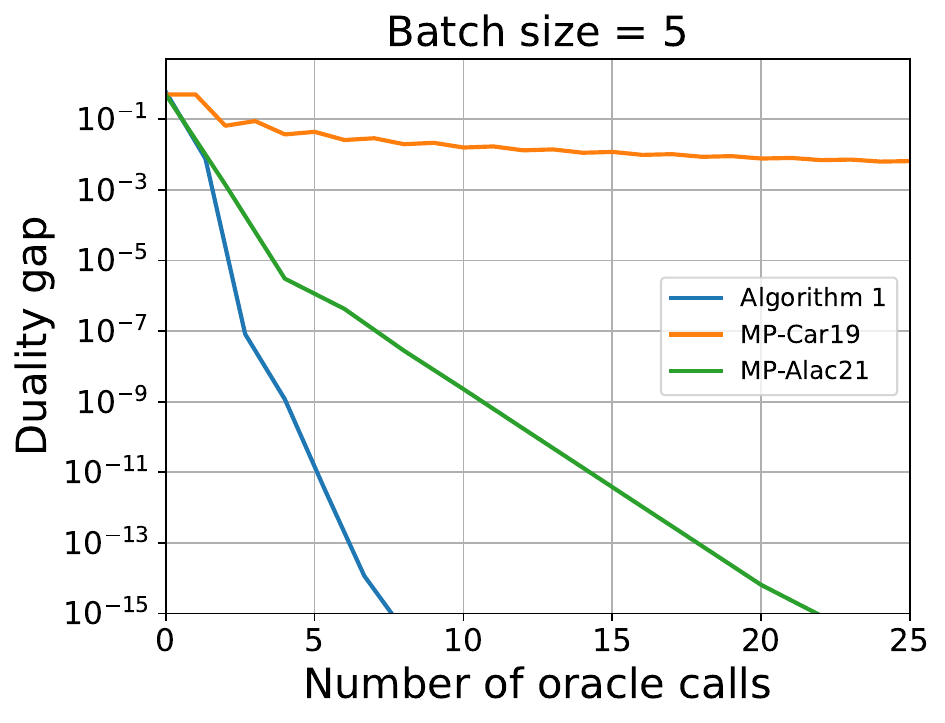}
\includegraphics[width=0.35\textwidth]{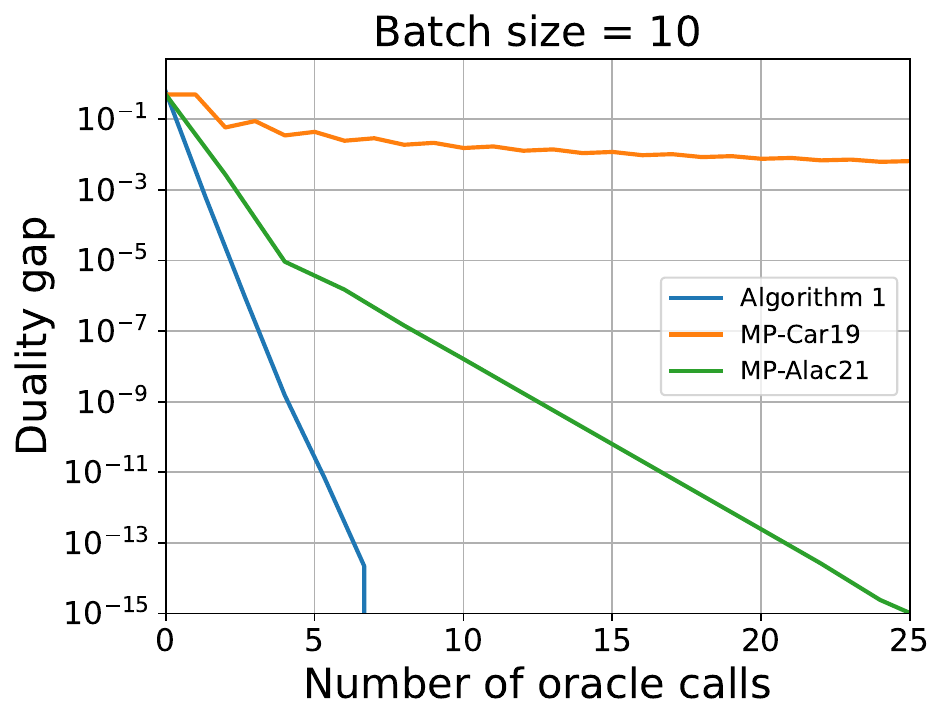}
\end{minipage}
\caption{Comparison of computational complexities for Algorithm~\ref{alg:new}(with adjusted parameters), \texttt{MP-Car19} (Algorithm 1+2 from \cite{carmon2019variance}), and \texttt{MP-Alac21} (Algorithm 2 \cite{alacaoglu2021stochastic}) with different batch sizes on first test matrix from \cite{nemirovski2013mini}.}
\label{fig:nem_updated}
\end{figure}

Additionally, we also notice that on given test matrices it is possible to improve convergence rates by adjusting parameters. Using grid search we achieve the parameters that guaranteed the best convergence. With these parameters we run algorithms and plot Figures \ref{fig:nem_updated} and \ref{fig:PB_updated}. The results of this experiment are similar: our algorithm provides better results than the competitors and the convergence rate does not depend on the batch size. 

\begin{figure}
\centering
\begin{minipage}[][][b]{\textwidth}
\centering
\includegraphics[width=0.35\textwidth]{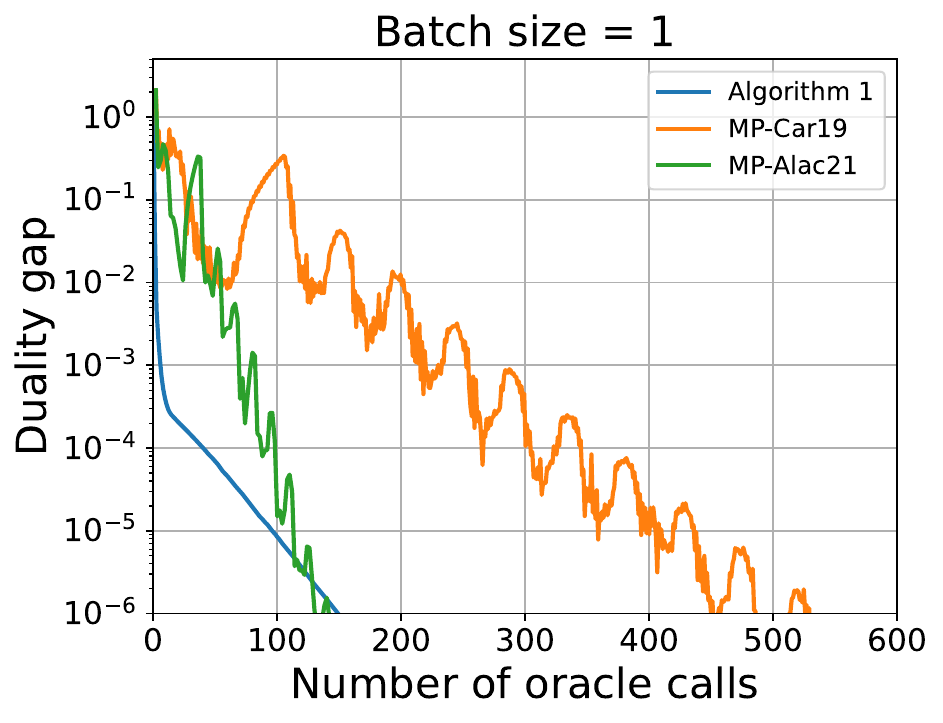}
\includegraphics[width=0.35\textwidth]{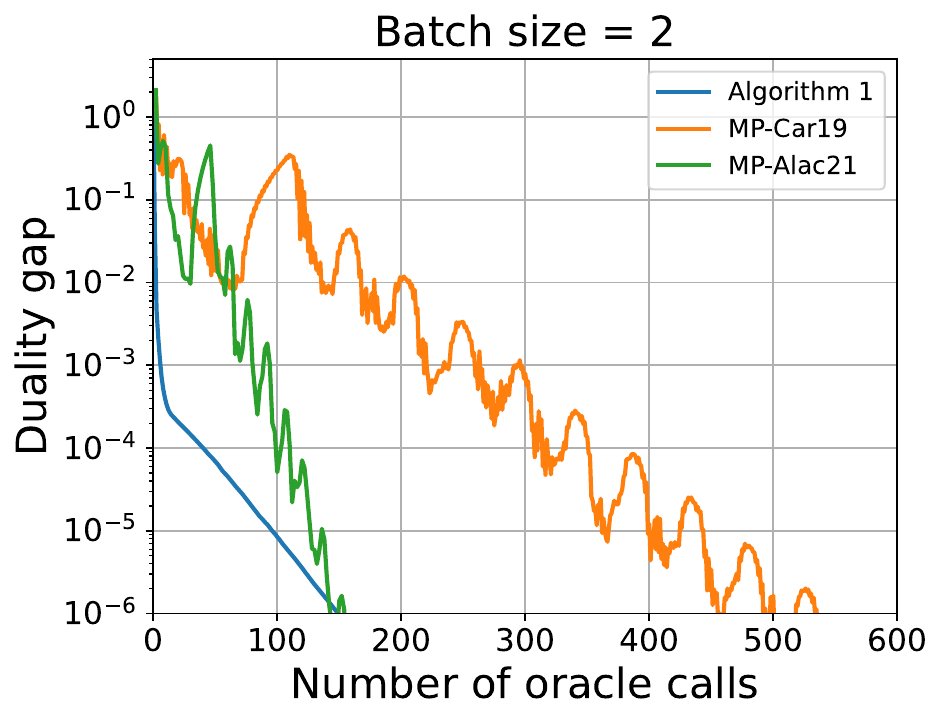}
\includegraphics[width=0.35\textwidth]{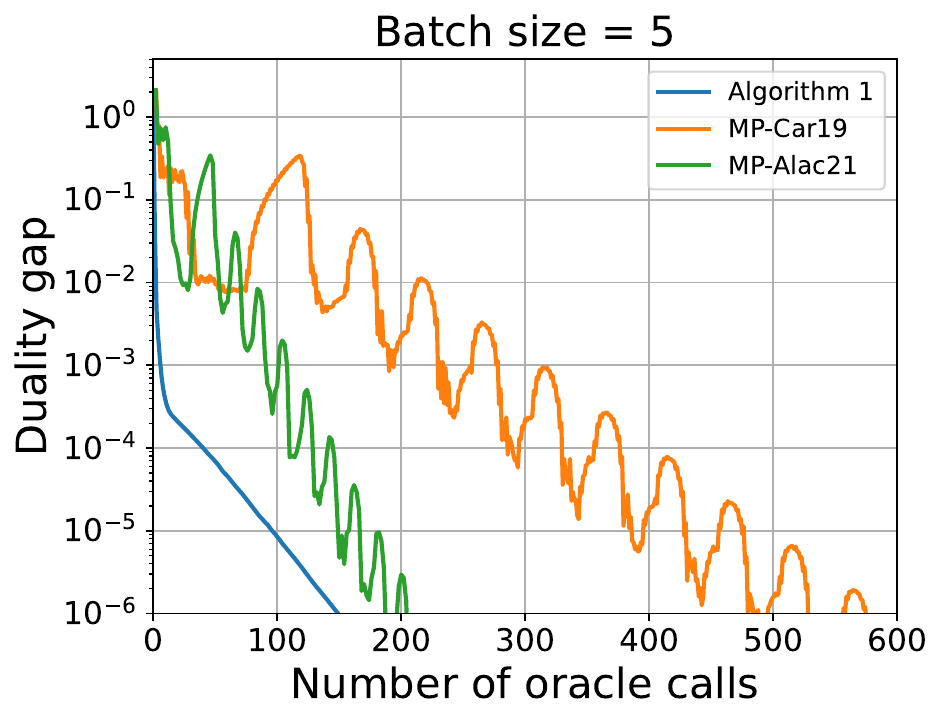}
\includegraphics[width=0.35\textwidth]{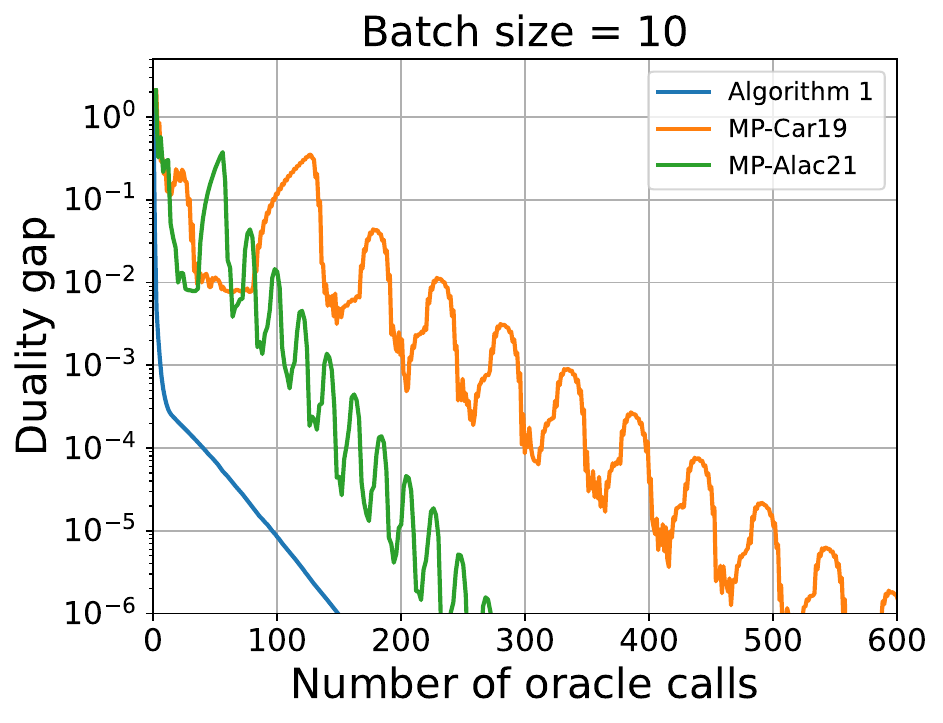}
\end{minipage}
\caption{Comparison of computational complexities for Algorithm~\ref{alg:new}(with adjusted parameters), \texttt{MP-Car19} (Algorithm 1+2 from \cite{carmon2019variance}), and \texttt{MP-Alac21} (Algorithm 2 \cite{alacaoglu2021stochastic}) with different batch sizes on policeman and burglar matrix from \cite{nemirovski2013mini}.}
\label{fig:PB_updated}
\end{figure}

\bibliographystyle{splncs04}
\bibliography{ltr}

\appendix
\section{Basic facts}

\begin{lemma} For all $u,v\in \mathbb{R}^n$ the Cauchy–Schwarz inequality holds:
\begin{align}
    |\left\langle u, v \right\rangle|^2 \leq \left\langle u, u \right\rangle\cdot \left\langle v, v \right\rangle . \label{ap:Cauchy–Schwarz}
\end{align}
\end{lemma}

\begin{lemma} For all $u,v\in \mathbb{R}^n$ , for any $p>0$ and $q>0$ such that $\frac{1}{p}+\frac{1}{q}$ the Young's inequality holds:
\begin{align}
    \langle u, v\rangle\leq \frac{u^p}{p}+\frac{u^q}{q}. \label{ap:Young}
\end{align}
\end{lemma}

\begin{lemma}[\cite{vandenberghe2022generalized}] For all $u$, $v$ and $w\in \mathcal{X}$ the  three point identity holds:
\begin{align}
    V(u,w) = V(u,v) + V(v,w) + \langle \nabla h(v)-\nabla h(w), u-v\rangle .
    \label{ap:3point}
\end{align}
\end{lemma}

\end{document}